\documentclass[preprint,12pt]{elsarticle}
\usepackage{color}
\usepackage{amsmath}
\usepackage{amsthm}
\usepackage{amsfonts}
\usepackage{amssymb}
\usepackage{graphicx, enumitem}
\usepackage{hyperref}
\usepackage{tabularx}

\def \N{\mathbb{N}}

\def \E{\mathbb{E}}
\def \R{\mathbb{R}}

\def \G{\mathcal{G}}
\def \A{\mathcal{A}}
\def \M{\mathcal{M}}
\def \eps{\varepsilon}
\def \oA{\overline{A}}
\def \G{\mathcal G}

\def \int{{\rm int\,}}

\def \B{\mathcal B}
\def \C{\mathcal C}

\def \U{\mathcal U}
\def \M{\mathcal M}
\def \G{\mathcal G}

\def \int{{\rm int\,}}

\def \B{\mathcal B}

\def \U{\mathcal U}
\def \M{\mathcal M}

\def \sgn {{\rm sgn}}

\providecommand{\U}[1]{\protect \rule{.1in}{.1in}}
\newtheorem{theorem}{Theorem}[section]

\newtheorem{convention}[theorem]{Convention}
\newtheorem{corollary}[theorem]{Corollary}

\newtheorem{definition}[theorem]{Definition}
\newtheorem{example}[theorem]{Example}

\newtheorem{lemma}[theorem]{Lemma}
\newtheorem{notation}[theorem]{Notation}

\newtheorem{proposition}[theorem]{Proposition}
\newtheorem{remark}[theorem]{Remark}

\journal{Indagationes Mathematicae}

\begin{document}

\begin{frontmatter}



\title{The explicit formula for Gauss-Jordan elimination and error analysis\tnoteref{label1}}

\author[label2]{Nam Van Tran}
\ead{namtv@hcmute.edu.vn}
\fntext[label1]{Faculty of Applied Sciences, HCMC University of Technology and Education, Ho Chi Minh city, Vietnam}

\author[label3,label4]{J\'ulia Justino}
\ead{julia.justino@estsetubal.ips.pt}
\fntext[label3]{Set\'ubal School of Technology,
	Polytechnic Institute of Set\'ubal,	Set\'ubal, Portugal}
\author{Imme van den Berg\fnref{label4}\corref{cor1}}\cortext[cor1]{Imme van den Berg}
\ead{ivdb@uevora.pt}
\fntext[label4]{Research Center in Mathematics and Applications (CIMA), University of \'Evora, \'Evora, Portugal}
\begin{abstract}

The explicit formula for the elements of the successive intermediate matrices of the Gauss-Jordan elimination procedure for the solution of systems of linear equations is applied to error analysis. Stability conditions in terms of relative uncertainty and size of determinants are given such that the Gauss-Jordan procedure leads to a solution respecting the original imprecisions in the right-hand member. The solution is the same as given by Cramer's Rule. Imprecisions are modelled by scalar neutrices, which are convex groups of (nonstandard) real numbers. The resulting calculation rules extend informal error calculus, and permit to keep track of the errors at every stage.

\end{abstract}



\end{frontmatter}



\section{Introduction}

In the present article we study imprecise systems of linear equations. The imprecisions occurring in the coefficient matrix and the right-hand member of a system of linear equations are modelled by convex subgroups of the nonstandard reals,  called \emph{(scalar) neutrices}. The vagueness is reflected by the invariance under some additions, a formalization of the Sorites property \cite{sorites, sorites2}; we were inspired by the functional neutrices of Van der Corput's Theory of Neglecting \cite{Van der Corput}. The setting within the real number system enables individual treatment of the imprecisions and a straightforward calculus modelling error-propagation.

Stability conditions for  systems of linear equations are formulated. In particular, the relative imprecisions of elements of the coefficient matrix $ \A $, when compared to $ \det(A) $, should be at most of same  order as the relative imprecision of the right-hand member, and $ \det(A) $ should also be not too small. We derive that  each Gauss-Jordan operation transforms a stable system into a stable system and that there is no significant blow-up of the imprecisions. The Main Theorem (Theorem \ref{maintheorem}) states that the Gauss-Jordan procedure solves a stable system within the bounds given by the imprecisions in the right-hand member, leading to the same outcome as  Cramer's Rule.

Within nonstandard analysis a neutrix is usually an external set. {\em External numbers} are sums of a real number and a neutrix. They give rise to the algebraic structure of a \emph{Complete Arithmetical Solid} \cite{DinisBergax}. This structure is weaker than a field, being based on additive and multiplicative semigroups, with a distributive law which is valid under some conditions \cite{Dinischar}. Still the structure is completely ordered, with a Dedekind completeness property and an Archimedean property, and is sufficiently strong to enable rather straightforward algebraic calculations \cite{Koudjeti Van den Berg}\cite{DinisBerg}, while common techniques and operations of linear algebra and matrix calculus remain valid to a large extend \cite{Jus} \cite{NamImme} \cite{Imme nam 3}. This may be observed  also in the present article.

It follows from the above that the scalar neutrices allow for a stronger algebraic structure than Van der Corput's neutrices, which is partly due to the absence of functional dependence. The last section of the present article contains a result which may be seen to fall within Van der Corput's program of Ars Negligendi: when we recognize a system as being stable, we may as well solve a simpler system, neglecting all terms in the coefficient matrix contained in its biggest neutrix.

There is an extensive literature on error analysis for the Gaussian and Gauss-Jordan elimination procedure, see e.g. \cite{Wilkinson}, \cite{Peter}, \cite{Ikramov} and \cite{Scott}, \cite{George}, which contain many more references.
Often the approach is asymptotic, as a function of the number of variables $ m $. Some key-notions are the \emph{growth factor} 
$\rho \equiv \dfrac{\max\limits_{i,j,k} |a^{(k)}_{i,j}|}{\max\limits_{i,j}|a_{i,j} |}$, where $ k\leq m $ and $ [a^{(k)}_{i,j}] $ is the $ k $-th intermediate matrix, and the \emph{condition number} in the form of the product of norms $\rm{cond}(A)\equiv\|A\|.\|A^{-1}\|$.

Here we choose a non-asymptotic approach taking $ m $ standard. The principal tools in our setting are explicit formulas for the elements of the transition matrices \cite{Li},\cite{NJI1} and estimates of determinants and its principal minors; this  seems somewhat natural, since by Cramer's Rule the solution of linear systems is stated in the form of quotients of determinants,  and the Gauss-Jordan operations are carried out with quotients of minors; we point out that there exists a relationship between the orders of magnitude of determinants and its principal minor, see Subsection \ref{excalculus} .

This article has the following structure. Section \ref{sectionGJ} recalls some basic properties of nonstandard analysis, and of neutrices and external numbers. Also some notions and notations are given for the Gauss-Jordan operations, matrices with external numbers and systems of linear equations with external numbers (flexible systems).  We define the notion of stability, and formulate two principal theorems, the first stating that stability is respected by the Gauss-Jordan operations, and the second indicating the solution sets of flexible systems. Section \ref{examples}  presents examples illustrating the principal theorems and the role of their conditions. Section \ref{preliminary} recalls useful properties of the calculus of external numbers and the explicit expressions for the elements of the transition matrices.
In Section \ref{lemmas} the impact of each step of the Gauss-Jordan procedure on the size of the neutrices is shown. These results and the generalization of Cramer's Rule proved in Section \ref{Cramer} allow us to prove the Main Theorem  in Section \ref{proofs}.  In Section \ref{neglection of terms} we define equivalent systems, having the same solutions, and show in the case of stability a given system may be substituted by a simpler equivalent system; this is illustrated numerically.

\section{Backgrounds and main theorems}\label{sectionGJ} 
We start with some background on Nonstandard Analysis in  Subsection \ref{nonstandard}. In Subsection \ref{neutrices} we recall the notions of neutrix and of external number used to model the imprecisions. In Subsection \ref{GaussJordan} we introduce some notions and notations with respect to the Gauss-Jordan operations, which we will effectuate in the form of matrix multiplications. Subsection \ref{matrices} contains notions and notations with respect to matrices and matrix operations. In Subsection \ref{flexible systems} we recall the definition of  flexible systems of linear equations, with a slight modification, and introduce a notion of stability. In  Subsection \ref{mainresults} we state the main theorems, the first saying that the Gauss-Jordan elimination procedure transforms a stable system into a stable system, and the second saying that stable systems maybe solved both by Cramer's rule and Gauss-Jordan elimination, leading to equal solutions. 

\subsection{Nonstandard Analysis}\label{nonstandard}

We adopt the axiomatic form of nonstandard analysis Internal Set Theory $IST$ of \cite{Nelson}; an important feature is that, next to the standard numbers, infinitesimals and infinitely large numbers are already present within the ordinary set of real numbers $\mathbb{R} $. We use only bounded formulas, and then neutrices and external numbers are well-defined external sets in the  extension $HST$ of a bounded form of $ IST $ given by Kanovei and Reeken in \cite{Kanovei}. For introductions to $ IST $ we refer to e.g. \cite{Dienerreeb}, \cite{Dienernsaip} or \cite{Lyantsekudryk} and for introductions to external numbers and illustrative examples we refer to \cite{Koudjeti Van den Berg}, \cite{DinisBergax} or \cite{DinisBerg}; the latter contains an introduction to a weak form of nonstandard analysis sufficient for a practical understanding of our approach. An important tool is \emph{External induction} which permits induction for all $ IST $-formulas over the standard natural numbers. 

A real number is \emph{limited} if it is bounded in absolute value by a standard natural number, and real numbers larger in absolute value than all limited numbers are called \emph{unlimited}. Its reciprocals, together with $ 0 $, are called \emph{infinitesimal}. \emph{Appreciable} numbers are limited, but not infinitesimal. The set of limited numbers is denoted by $\pounds$, the set of infinitesimals by $\oslash$, the set of positive unlimited numbers by $ \not\hskip -0.22cm \infty $ and the set of positive appreciable numbers by $ @ $; these sets are all external. 

\subsection{Neutrices and external numbers}\label{neutrices}



\begin{remark}
	Throughout this article we use the symbol $ \subseteq $ for inclusion and $ \subset $ for strict inclusion.
\end{remark}
\begin{definition}A \emph{(scalar) neutrix} is an additive convex subgroup of $\R$. An {\em external number}  is the Minkowski-sum of a real number and a neutrix. 
\end{definition}

So each external number  has the form $\alpha=a+A=\{a+x|x\in A\}$, where $A$ is called the {\em neutrix part} of $\alpha$, denoted by $N(\alpha)$, and $a\in \R$ is called a {\em representative}  of $\alpha$. 

In classical analysis the only neutrices are $\{0\}$ and $\R$, but in Nonstandard Analysis there are many more neutrices, all external sets. Examples are $\oslash$ and $\pounds$, for the sum of two infinitesimals is infinitesimal, and the sum of two limited numbers is limited. Let $\varepsilon\in \R$ be a positive infinitesimal. Other examples of neutrices are $\varepsilon\pounds$, $\varepsilon \oslash $,  $M_{\varepsilon}\equiv\displaystyle\bigcap_{st(n)\in \N}[-\varepsilon^n, \varepsilon^n]=\pounds\varepsilon^{\not\infty}$ and $\mu_{\varepsilon}\equiv\displaystyle\bigcup_{st(n)\in \N}[-e^{-1/(n\eps)}, e^{-1/(n\eps)}]=\pounds e^{-@/\eps}$; as groups they are not isomorphic. 
Identifying $ \{a\} $ and $ a $, the real numbers are external numbers with $ N(\alpha)=\{0\} $. We call $\alpha$ {\em zeroless} if $0\not\in \alpha$,  and {\em neutricial} if $ \alpha=N(\alpha) $.

Let $ N $ be a neutrix. Clearly $ \pounds N =N$. An \emph{absorber} of $N $ is a real number $a $
such that $aN\subset N$. No appreciable number is an absorber of any neutrix, and in the examples above the infinitesimal number $%
\varepsilon $ is an absorber of $\pounds$ and $\oslash $, but not of   $M_{\varepsilon} $ and $ \mu_{\varepsilon} $. Neutrices are ordered by inclusion, and if the neutrix $ A$ is contained in the neutrix $ B$, we may write $ B=\max\{A,B\} $.

Notions as \emph{limited}, \emph{infinitesimal} and \emph{absorber} may be extended in a natural way to external numbers.

The collection of all neutrices is not an external set in the sense of \cite{Kanovei}, but a definable class, denoted by $\mathcal{N}$. Also the external  numbers form a class, denoted by $ \E $. 

The rules for addition, subtraction, multiplication and division of external numbers of Definition \ref{defnam} below are in line with the rules of informal error analysis \cite{Taylor}. Here they are defined formally as  Minkowski operations on sets of real numbers.

\begin{definition}\label{defnam}\rm Let $ a,b\in \R $, $A, B$ be neutrices and $\alpha=a+A, \beta=b+B$ be external numbers.
	\begin{enumerate}
		\item \label{cong} $\alpha\pm\beta=a\pm b+A+B= a+b+\max\{A,B\}$.		
		\item \label{nhan}  $\alpha \beta=ab+Ab+Ba+AB=ab+\max\{aB, bA, AB\}.$
		\item \label{1/alpha} If $\alpha$ is zeroless, $\dfrac{1}{\alpha}=\dfrac{1}{a}+\dfrac{A}{a^2}.$		
	\end{enumerate}
\end{definition}

If $\alpha$ or $ \beta$ are zeroless, in Definition \ref{defnam}.\ref{nhan} we may neglect the neutrix product $ AB $. Definition \ref{defnam}.\ref{1/alpha} does not permit to divide by neutrices. However,  division of neutrices is allowed in terms of division of groups.

\begin{definition}
	Let $A, B\in \mathcal{N}$. Then we define  
	$$ A:B=\{c\in \R\ |\ cB\subseteq A\}.$$
\end{definition}

An order relation is given as follows.

\begin{definition} Let $ \alpha,\beta \in \E$. We define 
	\begin{equation*}
		\alpha\leq\beta\Leftrightarrow\forall a\in \alpha \exists b\in \beta(a\leq b).
	\end{equation*}
	If $ \alpha \cap \beta =\emptyset $ and $ \alpha\leq\beta $, then $ \forall a\in \alpha \forall b\in \beta(a<b) $ and we write $ \alpha<\beta $.
\end{definition}

The relation $ \leq  $ is an order relation indeed, and compatible with the operations, with some small adaptations \cite{Koudjeti Van den Berg}\cite{DinisBerg}. The inverse order relation is given by
\begin{equation*}
	\alpha\geq\beta\Leftrightarrow\forall a\in \alpha \exists b\in \beta(a\geq b),
\end{equation*}
and $ \alpha>\beta $ if $ \forall a\in \alpha \forall b\in \beta(a>b) $.
Clearly $ \alpha<\beta $ implies $ \beta>\alpha $. However, both $ \oslash\leq \pounds $ and $ \oslash\geq \pounds $ hold. External numbers $ \alpha $ such that $ 0\leq \alpha $ are called \emph{non-negative}.

The \emph{absolute value} of an external number $\alpha=a+A$ is defined by $|\alpha|=|a|+A$. Notice that this definition does not depend on the choice of the representative of $\alpha.$

By the close relation to the real numbers, practical calculations with external numbers tend to be quite straightforward, this may be verified on the examples of Section \ref{examples}. Some care is needed with distributivity, see Subsection \ref{excalculus}.
A full list of axioms for the operations on the external numbers has been given in  \cite{DinisBergax} and \cite{DinisBerg}. The resulting structure has been called a \emph{Completely Arithmetical Solid $(CAS)$}. A Completely Arithmetical Solid relates a completely regular commutative additive semigroup and a completely regular commutative multiplicative semigroup by the modified distributive law of Theorem \ref{tcopti2}, it has a total order relation with a generalized Dedekind property and contains two built-in models for the natural numbers. The results of the present article use only calculatory properties of nonstandard analysis, and could also have been presented in the setting of a $CAS$, and then the neutrices, external numbers and the $CAS$ itself are true sets.

\subsection{Gauss-Jordan operations \label{GaussJordan}}

The Gauss-Jordan operations will be effectuated by multiplications by elementary matrices. These matrices will have real coefficients. This reflects the common practice to use relatively simple numbers for these operations, and in this way more algebraic laws are respected. Below we give notations for the Gauss-Jordan procedure and the intermediate matrices.

We consider here only square matrice and denote by $\M_{ n}(\mathbb{R})$ the set of all $n\times n$
matrices over the field $\mathbb{R}$, with $ n\in \N, n\geq 1 $.

\begin{definition}\label{defgau}
	Let $\A=[a_{ij}]_{n\times n}\in \M_{n}(\R)$  be non-singular. For every $ q $ with $1\leq q\leq 2n$, the \emph{Gauss-Jordan operation matrix} $\mathcal{G}_{q}$ and the intermediate matrix $ \A^{(q)} $ are defined as follows. 
	
Let $\mathcal{G}_{0}$ be the $ n\times n $ identity matrix $I_{n}$ and $\A^{(0)}=\A$.
	Assuming that $ \mathcal{G}_{2k}$ and $\A^{(2k)}=[a^{(2k)}_{ij}]_{n\times n} $ are defined for $ k<n $, we also assume that $ a^{(2k)}_{k+1k+1} \neq 0$. Then $\mathcal{G}_{2k+1}=%
	\begin{bmatrix}
		g_{ij}^{(2k+1)}%
	\end{bmatrix}%
	_{n\times n}$, where  
	\begin{equation}\label{defgodd}
		g_{ij}^{(2k+1)}=%
		\begin{cases}
			\quad 1 & \mbox{ if }i=j\not=k+1 \\ 
			\quad 0 & \mbox{ if }i\not=j \\ 
			\dfrac{1}{a^{(2k)}_{k+1k+1}} & \mbox{ if }i=j=k+1%
		\end{cases},	
	\end{equation}
	leading to 
	\begin{equation}\label{defakodd}
		\A^{(2k+1)}=\mathcal{G}_{2k+1}\A^{(2k)}\equiv[a^{(2k+1)}_{ij}]_{n\times n}, 	 
	\end{equation}
	and 	$	\mathcal{G}_{2k+2}=%
	\begin{bmatrix}
		g_{ij}^{(2k+2)}%
	\end{bmatrix}%
	_{n\times n}$, where 
	
	\begin{equation}\label{defgeven}	g_{ij}^{(2k+2)}=%
		\begin{cases}
			0 & \mbox{ if }j\not \in \{i,k+1\} \\ 
			1 & \mbox{ if }j=i \\ 
			-a^{(2k+1)}_{ik+1}
			\medskip  &  \mbox{ if }i\neq k+1,j=k+1%
		\end{cases},	
	\end{equation}
	resulting in 
	\begin{equation}\label{defakeven}
		\A^{(2k+2)}=\mathcal{G}_{2k+2}\A^{(2k+1)}\equiv[\alpha^{(2k+2)}_{ij}]_{n\times n}. 	 
	\end{equation}
\end{definition}	
The matrix of odd order $ \mathcal{G}_{2k+1} $ corresponds to the multiplication of row $ k+1$ of $ \A^{(2k)} $ by $1/ a^{(2k)}_{k+1k+1} $,  and the matrix of even order $ \mathcal{G}_{2k+2} $ corresponds to transforming the entries of column $ k $ of $ \A^{(2k+1)} $ into zero, except for the entry $ a^{(2k+1)}_{k+1k+1}(=1) $.

Up to changing rows and columns we may always assume that the pivots $ a^{(2k)}_{k+1k+1} $ are non-zero. In fact they may always be chosen to be maximal, which is numerically desirable. The properties in question are a consequence of the next definition and propositions. We introduce first a notation for minors, taken from \cite{Gantmacher}.
\begin{notation}
	\label{kyhieu2}Let $\A\in \M_{n}(\mathbb{R})$. For each $k\in \N $ such that $1\leq k\leq n$, let $1\leq i_{1}<\dots <i_{k}\leq m$ and $1\leq
	j_{1}<\dots <j_{k}\leq n$. 
	
	\begin{enumerate}
		\item We denote the $k\times k$ submatrix  of $\A$ consisting of the rows with indices $\{i_{1},\dots ,i_{k}\}$ and columns with indices $\{j_{1},\dots ,j_{k}\}$ by $\A^{i_{1}\dots i_{k}}_{j_{1}\dots j_{k}}$. 
		
		\item We denote the corresponding $k\times k$ minor by 		
		\begin{equation}\label{defminor}
			m^{i_{1}\dots i_{k}}_{j_{1}\dots j_{k}}=\det\left( \A^{i_{1}\dots i_{k}}_{j_{1}\dots j_{k}}\right). \end{equation}
		
		\item \textrm{For $ 1\leq k\leq \min\{m,n\}$ we may denote the principal minor of order $ k $ by $m_{k}=m_{1\cdots k}^{1\cdots k}$. We define formally $m_{0}=1$}.
	\end{enumerate}
\end{notation}

\begin{definition}
	\label{matrixofgaussJordanelimination} Assume $\A\in \M_{n}(\mathbb{R}%
	) $. Then $\A$ is called \emph{properly arranged}, if $|a_{ij}| \leq |a_{11}| $ for $1\leq i\leq n$ and $1\leq j\leq n$ and $\left \vert m_{1\cdots k j}^{1\cdots ki}\right \vert \leq \left \vert m_{k+1}\right \vert$ for every $ k $ such that $1\leq k\leq n-1$, whenever $k+1\leq i\leq n$ and $k+1\leq j\leq n$. We say that $ \A $ is \emph{diagonally eliminable}, if $ m_{k}\not=0 $ for $1\leq k\leq n$.
\end{definition}

\begin{proposition}\label{propar}	
	Let $ n\geq 1 $. Let $\A=[a_{ij}]_{n\times n}\in \M_n(\R)$. By if necessary changing rows and columns $\A$ can be properly arranged. 
\end{proposition}
The proof of this proposition is straightforward. If $ \A $ is non-singular, by Proposition \ref{diagonally} the matrix is automatically diagonally eliminable, so we assume without restriction of generality that this is always the case.

\begin{proposition}\label{diagonally}\cite{NJI1} 
	Let $ n\geq 1 $. Let $\A=[a_{ij}]_{n\times n}\in \M_n(\R)$ be  non-singular and properly arranged. Then $\A$ is diagonally eliminable.
\end{proposition}
\begin{definition}\label{procedure}Let $ n\geq 1 $. Let $\A=[a_{ij}]_{n\times n}\in \M_n(\R)$ be  non-singular and properly arranged. Then we call the sequence $ \A,\A^{(1)},\cdots,\A^{(2n)} $ the \emph{Gauss-Jordan procedure} and we write $ \G=\mathcal{G}_{2n}({G}_{2n-1}\cdots(\mathcal{G}_{2}\mathcal{G}_{1}))$.
\end{definition}
\subsection{Matrices with external numbers}\label{matrices}
We denote by  $\M_{m,n}(\E)$ the class of all $m\times n$ matrices
\begin{equation}\label{forextmat}
	\A=\begin{bmatrix}
		\alpha_{11} & \alpha_{12}& \cdots & \alpha_{1n}\\
		\vdots & \vdots&  \ddots & \vdots \\
		\alpha_{m1} & \alpha_{m2} & \cdots & \alpha_{mn}
	\end{bmatrix},
\end{equation}
where $\alpha_{ij}=a_{ij}+A_{ij}\in \E$ for $ 1\leq i\leq m, 1\leq  j\leq n $; we always suppose that $m,n\in \N, m,n\geq 1$ are standard. The matrix $ \A $ is called an \emph{external matrix} and we use the common notation $\A=[\alpha_{ij}]_{m\times n}$. 	A matrix $\mathcal{A}\in \M_{m,n}(\E)$ is said to be \emph{neutricial} if all of its entries are neutrices. With respect to  \eqref{forextmat} the matrix $P=[a_{ij}]_{m\times n}\in \M_{m, n}(\mathbb{R})$ is called a \emph{representative matrix} and the matrix $A=[A_{ij}]_{m\times n}$ the \emph{associated neutricial matrix}. If $ m=n $ we may write $\M_{n}(\E)$ instead of $\M_{m,n}(\E)$. A matrix $ \A\in \M_{n}(\E) $ with representative matrix equal to the identity matrix $ I_{n} $ and associated neutricial matrix contained in $ [\oslash]_{n\times n} $ is called a \emph{near-identity matrix}, and is denoted by $ \mathcal{I}_{n}$. 

For $ \A,\B\in \M_{m\times n}(\E)$ we write $\A\subseteq \B$ if $\alpha_{ij}\subseteq \beta_{ij}$ for all $ i,j $ such that $ 1\leq i\leq m, 1\leq j\leq n$.

\begin{notation}
	
	Let $\A=[\alpha_{ij}]_{n\times n}\equiv [a_{ij}+A_{ij}]_{n\times n}\in \M_{n}(\E)$. We define 
	$$  |\overline{\alpha}|=\max\limits_{\substack{1\leq i,j\leq n}} \left| \alpha_{ij}\right|,\overline{A}=\max\limits_{\substack{1\leq i,j\leq n}} A_{ij}.$$ 
\end{notation}

\begin{definition}\label{deflired}
	An external matrix $\mathcal{A}$ is said to be \emph{limited} if $ \overline{\alpha}\subset \pounds $ and \emph{reduced} if $ \overline{\alpha}=\alpha_{11} $ and $ \alpha_{11}=1+A_{11} $, with $ A_{11}\subseteq\oslash $, while all other entries have representatives which in absolute value are at most $1$. 
\end{definition}
By the second part of Definition \ref{deflired} a reduced external matrix always has  a reduced representative matrix.

For $ \mathcal{A}\in \M_{n}(\E) $, the determinant $%
\Delta \equiv \det(\mathcal{A}) \equiv d+D$ is defined in the usual way through sums of signed products \cite{Jus}.

\begin{definition}\label{defmatnonsin}
	Let $ \A\in \M_{n}(\E) $. Then $\mathcal{A}$ is called {\em non-singular} if $\Delta $ is zeroless.
\end{definition}
Observe that a representative matrix of a non-singular matrix $ \A $ is always non-singular. It is not true in general that $ \det(\mathcal{A}) $ is equal to the set of determinants of representative matrices. This is shown in
\cite{Imme nam 3}, which contains an overview of the calculus of matrices with external numbers and its determinants.

Let $ \A=[\alpha_{ij}]_{n \times n}  $ be an external matrix. The Gauss-Jordan operations will always be effectuated using the elements of a representative matrix $ P=[a_{ij}]_{n \times n} $. In particular the notions of properly arranged and diagonally eliminable are defined by reference to representative matrices. 

\begin{definition}\label{defpropar}
	Let $\mathcal{A}\in \M_{n}(\E)$ be reduced and non-singular. \begin{enumerate}
		\item \label{defpp1}We say that $ \A $ is \emph{properly arranged} if it has a properly arranged representative matrix $ P $. In this case we say that $ \A $ is \emph{properly arranged with respect to $ P $}.
		\item \label{ky hieu phep toan gauss} If $ \A $ has a diagonally eliminable representative matrix $P$, we say that $ \A $ is \emph{diagonally eliminable with respect to  $P$}.  
	\end{enumerate}
\end{definition}
Because $ \A $ has a reduced representative matrix, by Proposition \ref{propar} we may assume without restriction of generality that a properly arranged representative matrix $ P $ is reduced. The matrix $ P $ is non-singular, for it is a representative matrix of the nonsingular matrix  $ \A $. Hence $ P $ is diagonally eliminable by Proposition \ref{diagonally}.

Definition \ref{defGJext} extends the notions of Gauss-Jordan operations matrix and intermediate matrix to matrices of external numbers.
\begin{definition}\label{defGJext} Let $\A\in \M_{n}(\E)$ be diagonally eliminable with respect to a representative matrix $P$. For $ 1\leq q \leq 2n $ we denote the $ q^{th} $ Gauss-Jordan operations matrix by $\G^{P}_{q}$ and we write 
	$\G^P=\G^{P}_{2n}(\G^{P}_{2n-1}\cdots(\G^{P}_{2}\G^{P}_{1}))$.
\end{definition}
We will see that under the condition of stability of Definition \ref{defrelun} below, the result of the Gauss-Jordan procedure does not depend on the choice of the representative matrix and we may simply write $\G^P=\G$. We may also write $\G_{2q}^P=\G_{2q}$ for $1\leq q\leq 2n$ if there is no ambiguity on the representative matrix $ P $. Then we also adopt the notation of Definition \ref{defgau} for the intermediate matrices, i.e. we have
\begin{align}\label{defaqext}
	\G_q^P(\G_{q-1}^P \cdots (\G_1^P\A)) &= \G_q(\G_{q-1} \cdots (\G_1\A))\equiv\A^{(q)}\notag \\
	&\equiv[\alpha^{(q)}_{ij}]_{n\times n}\equiv[a^{(q)}_{ij}+A^{(q)}_{ij}]_{n\times n}=P^{(q)}+[A^{(q)}_{ij}]_{n\times n}.
\end{align}
The \emph{Gauss-Jordan procedure} applied to the matrix $ \A $ is the sequence $ \A,$ $\A^{(1)},\dots,\A^{(2n)} $. We will see that for stable matrices the last matrix is a near-identity matrix.


We recall the notion of relative uncertainty for matrices from \cite{Jus}, and use it to define stable matrices. 


\begin{definition}\label{defrelun}
	Let $\mathcal{A}=[\alpha_{ij}]_{n\times n}\in \M_{n}(\E)$ be a limited non-singular matrix. The  \emph{relative uncertainty} $ R(\mathcal{A }) $ is defined by $ R(\mathcal{A })=\overline{A}/\Delta $. The matrix $ \A $ is called \emph{stable} if $  R(\mathcal{A })\subseteq \oslash$.
\end{definition}

The biggest neutrix $ \overline{A} $ occurring in a limited matrix $ \A $ is always contained in $ \oslash $, but if $ \Delta $ is infinitesimal, for the matrix to be stable, the entries need to be sharper.

\subsection{Flexible systems and stability}\label{flexible systems}

We recall the definition of flexible systems of linear equations of \cite{Jus} in  a slightly modified form, and show equivalence with the earlier definition. For the particular case of square non-singular systems we define a notion of stability, implying that the Gauss-Jordan operations give rise to at most a moderate increase of imprecisions.

\begin{definition}\label{extvec}
	Let $ n\in \N $ be standard and $ \xi_1,\dots,\xi_n $ be external numbers. Then $\xi\equiv (\xi_1,\dots,\xi_n)^{T}$ is called an \emph{external vector}. For $ 1\leq i\leq n $, let $ \xi_i=x_{i}+X_{i} .$ Then $ x\equiv(x_{1},\dots,x_{n})^{T} $ is called a \emph{representative vector} and $ X\equiv(X_{1},\dots,X_{n})^{T} $ is called the \emph{associated neutricial vector}, i.e. $ \xi=x+X.$
\end{definition}

\begin{definition}
	\label{defflexiblesystem} A \emph{flexible system} is a system of inclusions	
	\begin{equation}  \label{hpttr}
		\left \{%
		\begin{matrix}
			\alpha_{11} x_1+ & \alpha_{12}x_2+ & \cdots & +\alpha_{1n}x_n & 
			\subseteq \beta_{1} \\ 
			\vdots & \vdots & \ddots & \vdots & \vdots \\ 
			\alpha_{m1} x_1+ & \alpha_{m2}x_2+ & \cdots & +\alpha_{mn}x_n & 
			\subseteq \beta_{m}%
		\end{matrix}%
		\right.,
	\end{equation}
	where $m,n$ are standard natural numbers, $ x=(x_{1},\dots,x_{n})\in \R^{n} $ and  $\alpha_{ij}\equiv a_{ij}+A_{ij}$ and $\beta_{i}\equiv b_i+B_{i}$ are external numbers for $ 1\leq i\leq m $ and $ 1\leq j\leq n $. We denote the matrix $%
	[\alpha_{ij}]_{m\times n} $ by $\mathcal{A}$, the representative matrix $ [a_{ij}]_{m\times n}$ by $ P $, the associated neutricial matrix $ [A_{ij}]_{m\times n}$ by $ A $ and the external vector $(\beta_{i},\dots,\beta_{m})^T $ by $\mathcal{B} $.  A vector $x$ is called an \emph{admissible solution} of the flexible system \eqref{hpttr} if it satisfies the system.
	
\end{definition}

The system \eqref{hpttr} is equivalent with the inclusion $ \mathcal{A}x\subseteq\mathcal{B} $, and usually is written in the matrix form $ \A|\B $.

In \cite{Jus} and \cite{NamImme} flexible systems with variables in the form of external numbers have been considered, i.e. systems of the form 
\begin{equation}  \label{hpttq}
	\left \{%
	\begin{matrix}
		\alpha_{11} \xi_1+ & \alpha_{12}\xi_2+ & \cdots & +\alpha_{1n}\xi_n & 
		\subseteq \beta_{1} \\ 
		\vdots & \vdots & \ddots & \vdots & \vdots \\ 
		\alpha_{m1} \xi_1+ & \alpha_{m2}\xi_2+ & \cdots & +\alpha_{mn}\xi_n & 
		\subseteq \beta_{m}%
	\end{matrix}%
	\right.,
\end{equation}
with an \emph{admissible solution} being an external vector $\xi$ satisfying the system. The systems \eqref{hpttq} and \eqref{hpttr} are equivalent, because of  Proposition \ref{extreal}.
\begin{proposition}\label{extreal}
	An external vector $ \xi $ is an admissible solution of \eqref{hpttq} if and only if every representative $ x $ of $\xi$ is an admissible solution of \eqref{hpttr}.
\end{proposition}
\begin{proof}
	Let $ \xi=(\xi_{1},\dots,\xi_{n})$. It is obvious that if the inclusions of \eqref{hpttq} are satisfied by $ (\xi_{1},\dots,\xi_{n}) $, they are also satisfied by any representative vector $ x=(x_{1},\dots,x_{n}) $. 
	
	Conversely, let $ 1\leq i\leq m $, and assume that $ \alpha_{i1} x_{1}+  \cdots  +\alpha_{in}x_{n}\subseteq \beta_{i} $ whenever $ x_{1}\in \xi_{1},\dots, x_{n}\in \xi_{n}$. Let $ t\in \tau\equiv\alpha_{i1} \xi_1+  \cdots  +\alpha_{in}\xi_n  $. It follows from the definition of the Minkowski operations that for all $ j $ with $ 1\leq j\leq n$ there exist $ a_{ij} \in \alpha_{ij}$ and $ \overline{x_{j}}\in \xi_{j} $ such that $ t= a_{i1} \overline{x_{1}}+  \cdots  +a_{in}\overline{x_{n}} $. Then $ t\in \beta_{i}  $. Hence $ \tau \subseteq \beta_{i}  $. 
	
	We conclude that the external vector $ \xi $ is an admissible solution of \eqref{hpttq} if and only if all its representative vectors are.	
\end{proof}
In the present article we only study the system \eqref{hpttr} with real variables, in the case $ m=n $.

We now introduce some terminology, in particular we carry over some of the notions on matrices of Section \ref{matrices} to systems of equations.

\begin{definition}\label{defsysA}
	The system $ \A|\B $ is called \emph{reduced} if $\mathcal{A}$ is reduced, \emph{homogeneous} if $ \B $ is a neutrix vector, \emph{upper homogeneous} if $ \overline{\beta} $ is a neutrix and \emph{uniform} if the neutrices of the right-hand side $B_{i}\equiv B $ are all the same. The system is called {\em non-singular} if $\mathcal{A}$ is non-singular, \emph{properly arranged} respectively \emph{diagonally eliminable} (with respect to a matrix of representatives $ P $) if $ \A $ is properly arranged respectively diagonally eliminable with respect to $ P $. 
\end{definition}

We recall the following notion of relative uncertainty of external vectors from \cite{Jus} and use it, together with the notion of relative uncertainty of matrices of Definition \ref{defrelun} to define stable systems.

\begin{definition}
	Let $ \mathcal{B}=(\beta_{1},\dots,\beta_{n})^{T} $ be an external vector. 	Let $ |\overline{\beta}|=\max\limits_{1\leq i\leq n} |\beta_{i}|$ and $\underline{B}=\min\limits_{1\leq i\leq n} B_{i}$. 
	
	If $ \overline{\beta} $ is zeroless, its  \emph{relative uncertainty} $ R(\mathcal{B})$ is defined by 
	
	\begin{equation*}	
		R(\mathcal{B})=\underline {B}\slash{\overline{\beta}}.	
	\end{equation*}	
	In the special case that $ \overline{\beta}=B $ for some neutrix $ B $, we define
	\begin{equation*}	
		R(\mathcal{B})=\underline{B}:B.	
	\end{equation*}	
\end{definition}
Observe that, whenever $ 1\leq i \leq n $, it holds that 
\begin{equation}\label{rbeta}
	R(\mathcal{B})\beta_{i}\subseteq B,\end{equation}
and that $ R(\mathcal{B})\subseteq\oslash $ if the system is not upper homogeneous.

\begin{definition}
	\label{dn Gauss-Jordan eliminable-new} Let $ \A\in \M_{n}(\E) $ be limited and non-singular. The system $ \A|\B $ is said to be \emph{stable} if 	
	\begin{enumerate}
		\item $ \A $ is stable. \label{dn1}
		\item $R(\mathcal{A}) \subseteq R(\mathcal{B})$. \label{dn2n}
		\item \label{dn4n} $\Delta$ is not an absorber of $\underline{B}$.
	\end{enumerate}
\end{definition}	

Condition \eqref{dn2n} expresses that in a sense the coefficient matrix is more precise than the right-hand member and Condition \eqref{dn4n} expresses that the determinant $\Delta $, which of course must be non-zero, should not be too small. Condition \eqref{dn1} originating from Definition \ref{defrelun} expresses that the neutrices in the coefficient matrix are small with respect to the determinant. Note that if $ R(\A)=\oA/\Delta\supset \oslash$ we have $R(\B) \supseteq R(\A)\supseteq\pounds  $, which means the system must be upper homogeneous. Then uniform systems, which form the  principal class of systems under consideration (see Convention \ref{conv} below), are homogeneous, which is very restrictive.


It will be seen that the notion of stability is respected by the steps of the Gauss-Jordan procedure, which possibbly lead to only a moderate increase of imprecisions in the coefficient matrix and leave the imprecision of the right-hand member invariant, finally resulting in the same imprecision for the solution. 

\begin{notation}\label{notpropar}
	Suppose that $\A\in \M_n(\E)$ and that the system  $ \A|\B $ is non-singular, properly arranged with respect to a matrix of representatives $ P $ and uniform.  We write $[B]=[B, B,\dots, B]^T$, $ \B^{(0)}=\B $ and  $ [B]^{(0)}=[B^{-1}]^{(0)}=[B] $. For $1\leq q\leq 2n$ we write	
	\begin{alignat*}{2}	
		\B^{(q)}&=\G^{P}_q(\G^{P}_{q-1}\cdots \G^{P}_1\B))\equiv [\beta^{(q)}_1,\dots, \beta^{(q)}_n]^{T}, \notag\\
		\left|\overline{\beta^{(q)}}\right|&=
		\max_{1\leq i\leq n}\left|\beta_i^{(q)}\right|,\notag\\
		[B]^{(q)}&=\G^{P}_q(\G^{P}_{q-1}\cdots \G^{P}_1[B])).\notag	
	\end{alignat*}
\end{notation} 
With these notations we see that after $ q $ Gauss-Jordan operations the system becomes $ \A^{(q)}|\B^{(q)} $, and for $ q=2n$ the Gauss-Jordan procedure ends with $ \G^{P}\A|\G^{P}\B $.

\subsection{Solutions sets and main results}\label{mainresults} 

We define solution sets in several ways. The Main Theorem states that under the conditions of stability they are all equal. We start with a theorem which is the principal tool for the proof of the Main Theorem, saying that the Gauss-Jordan operations respect the notion of stability. We will always suppose that Convention \ref{conv} holds.

\begin{convention}\label{conv}
	From now on we always suppose that the system  $ \A|\B $ is square, i.e. $\A\in \M_n(\E)$, and that the system is non-singular, reduced, 
	properly arranged with respect to a reduced matrix of representatives $ P $ and uniform.
\end{convention}

As for non-singular systems, only the condition of uniformity is restrictive. In the context of the Gauss-Jordan procedure the condition is essential, since the simple addition of equations may have the effect that the solution of the resulting system is no longer feasible for the original system, see Example \ref{ex1change}. By transforming a system with different neutrices $ B_{1},\dots,B_{n} $ in the right-hand side into the system with neutrices at the right-hand side equal to $ \underline{B} $, we get a uniform system, whose solutions are always feasible with respect to the original system.


The first theorem expresses that each Gauss-Jordan operation transforms a stable system into a stable system, and at the end we find a stable system with a coefficient matrix in the form of a near-identity matrix. As we will see its solution is simply the right-hand member.
\begin{theorem}	\label{lmqhnt}   Suppose that the flexible system $ \A|\B $ is properly arranged with respect to a representative matrix $P$ and stable.  Then 
	\begin{enumerate}
		\item \label{p1main} The intermediate system $ \A^{(q)}|\B^{(q)}  $ is stable for all $ q $ such that $0\leq q\leq 2n$. 
		
		\item \label{p2main} In particular  			
		$ \mathcal{G}^P\mathcal{A}|\G^P\mathcal{B} $ is stable, and $ \mathcal{G}^P\mathcal{A} $ is a near-identity matrix.
	\end{enumerate}
\end{theorem}

\begin{definition}\label{defsol}
	The \emph{solution} $ S $ of $ \A|\B $ is the (external) set of all real admissible solutions. If the Minkowski product $ \mathcal{A}S$ satisfies $\mathcal{A}S=\mathcal{B} $ we call $ S $ \emph{exact}.
\end{definition}

We now define the Gauss-Jordan solution and the Cramer solution.

\begin{definition}\label{defgauex}
	Assume the system $ \A|\B $ is properly arranged with respect to a  representative matrix $P$ of $\A$. The \emph{Gauss-Jordan solution} $ G^{P} $ of $ \A|\B $ with respect to $ P $ is defined by 
	\begin{equation}\label{forga}
		G^{P}=\left\{x\in \R^{n}|\big(\G^{P}(\A)\big)x\subseteq \mathcal{G}^{P}(\mathcal{B})\right\}. 
	\end{equation}
	If $ G^{P} $ does not depend on the choice of $ P $, we simply  call it the Gauss-Jordan solution, denoted by $ G $.
\end{definition}
\begin{definition}\label{defcra}
	Consider the system $ \A|\B $. Let $M_i$ be the matrix obtained from $\A$ by the substitution of the $%
	i^{th}$ column by the right-hand member $ \B$. Then the external vector
	
	\begin{equation}\label{forcra}
		\xi^{T}=\left(\dfrac{\det (M_1)}{\Delta},\ldots,%
		\dfrac{\det (M_n)}{\Delta}\right)^{T}
	\end{equation}
	is called the \emph{Cramer-solution} if every representative vector $ x $ satisfies $ \A|\B $. 
	
\end{definition}

\begin{theorem}[Main Theorem]\label{maintheorem} Assume the system $ \A|\B $ is stable, and properly arranged with respect to a representative matrix $P$ of $\A$. Let $ S $ be its solution. Then 
	\begin{equation}\label{formt}S=G=\G^{P}(\B)= \left(\dfrac{\det (M_1)}{\Delta},\ldots,	\dfrac{\det (M_n)}{\Delta}\right)^{T}.
	\end{equation}	
\end{theorem}

The solution of stable systems by Cramer's rule was shown in \cite{Jus} for non-homogeneous systems. 

\section{Examples}\label{examples}

We start with some straightforward applications of Theorem \ref{maintheorem}. Then we indicate some properties of flexible systems which are not shared by ordinary systems, and illustrate the role of the conditions of Theorem \ref{maintheorem}.

In example \ref{ex6} we verify first that the system is stable, then we show Gauss-Jordan procedure in some detail, searching for neutrices instead of zeros, to see at the end that the right-hand side is the solution indeed. Observe that the solution is given in the form of truncated expansions.
\begin{example}	\label{ex6}	
	Consider the  system 
	\begin{equation}\label{sy6}
		\left\{ 
		\setlength\arraycolsep{0pt}
		\begin{array}{ r >{{}}c<{{}} r >{{}}c<{{}} r  @{{}\subseteq {}} r  >{{}}c<{{}}r }
			\left( 1+\varepsilon ^{2}\oslash \right) x _{1} &+ &x _{2} &+& \left(
			1+\varepsilon ^{3}\pounds \right) x _{3}& 1
			+\varepsilon \oslash \\ 
			\left( 1+\varepsilon ^{3}\pounds \right) x _{1}&+&\left( -\frac{1}{2}+\varepsilon ^{2}\oslash
			\right) x _{2}&-&\frac{1}{2}x _{3}&  -2+\varepsilon \oslash \\ 
			\left( \frac{1}{2}\varepsilon +\varepsilon ^{3}\oslash \right) x _{1}&+&\frac{1}{2}x _{2}&+&\left(
			1+\varepsilon ^{2}\oslash \right) x _{3}& \varepsilon+\varepsilon \oslash %
		\end{array},
		\right.
	\end{equation}%
	where $\varepsilon $ is a positive infinitesimal. Let $\mathcal{A}$ be its matrix of coefficients and $\mathcal{B}$ be the right-hand member.
	The matrix is reduced and non-singular, with $\Delta =\det \mathcal{A}=-\frac{3}{4}+\varepsilon ^{2}\oslash $ 
	zeroless. 
	Let \begin{equation}\label{repmatrixex}P=\begin{bmatrix}
			1&1&1\\ 1& -\frac{1}{2} &-\frac{1}{2}\\
			\frac{1}{2}\eps &\frac{1}{2}&1
		\end{bmatrix}.
	\end{equation} 
	Then $P$ is a representative matrix of $\A$, and is properly arranged. Indeed, formula \eqref{defminor} is obvious for $k=1$, and is also satisfied for $k=2$  with 
	$m_2=m_{12}^{12}
	=-\frac{3}{2},\quad m_{13}^{12}
	=-\frac{3}{2}, \quad  m_{12}^{13}
	=\frac{1}{2}-\frac{1}{2}\eps, \quad  m_{13}^{13}
	=1-\frac{1}{2}\eps$. As a consequence, $\A$ is properly arranged.
	Because $ R(\A)=\varepsilon ^{2}\oslash\subseteq R(\B)=\varepsilon \oslash$ and $ \Delta B =(-\frac{3}{4}+\varepsilon ^{2}\oslash)\varepsilon \oslash=\varepsilon \oslash=B  $, the system is stable. 
	
	By Theorem \ref{dlng3} the solution may be obtained by the Gauss-Jordan procedure. It is given by \begin{equation}\label{sol1}
		S\equiv\left[ 
		\begin{array}{l}
			\xi _{1} \\ 
			\xi_{2} \\ 
			\xi_{3}%
		\end{array}%
		\right] =\left[ 
		\begin{array}{l}
			-1+\varepsilon \oslash \\ 
			4-3\varepsilon+\varepsilon \oslash \\ 
			-2+3\varepsilon+\varepsilon \oslash%
		\end{array}%
		\right],
	\end{equation}
	which we verify in detail. 	The second and third coordinate of $ S $ have the form of a truncated expansion. Also some expansions appear in the coefficients of the intermediate matrices. We get the following succession of stable systems:
	\begin{align*}
		\mathcal{A}|\mathcal{B} &=\left[ 
		\begin{array}{ccccc}
			1+\varepsilon ^{2}\oslash & 1 & 1+\varepsilon ^{3}\pounds  & | & 1+\varepsilon \oslash \\ 
			1+\varepsilon ^{3}\pounds  & -\frac{1}{2}+\varepsilon ^{2}\oslash & -\frac{1}{2} & | & -2+\varepsilon
			\oslash \\ 
			\frac{1}{2}\varepsilon +\varepsilon ^{3}\oslash & \frac{1}{2} & 1+\varepsilon ^{2}\oslash & | & 
			\varepsilon+\varepsilon \oslash%
		\end{array}%
		\right] \\
		&%
		\begin{tabular}{c}
			$\longrightarrow $ \\ 
			$L_{2}-L_{1}$ \\ 
			$L_{3}-\frac{1}{2}\varepsilon L_{1}$%
		\end{tabular}%
		\left[ 
		\begin{array}{ccccc}
			1+\varepsilon ^{2}\oslash & 1 & 1+\varepsilon ^{3}\pounds  & | & 1+\varepsilon \oslash \\ 
			\varepsilon ^{2}\oslash & -\frac{3}{2}+\varepsilon ^{2}\oslash & -\frac{3}{2}+\varepsilon ^{3}\pounds 
			& | & -3+\varepsilon \oslash \\ 
			\varepsilon ^{3}\oslash & \frac{1}{2}-\frac{1}{2}\varepsilon & 1-\frac{1}{2}\varepsilon +\varepsilon
			^{2}\oslash & | & \frac{1}{2}\varepsilon+\varepsilon \oslash%
		\end{array}%
		\right] \medskip \\	
		&%
		\begin{tabular}{c}
			$\longrightarrow $ \\ 
			$-\frac{2}{3}L_{2}$ \\ 
		\end{tabular}%
		\left[ 
		\begin{array}{ccccc}
			1+\varepsilon ^{2}\oslash & 1 & 1+\varepsilon ^{3}\pounds  & | & 1+\varepsilon \oslash \\ 
			\varepsilon ^{2}\oslash & 1+\varepsilon ^{2}\oslash & 1+\varepsilon ^{3}\pounds  & |
			& 2+\varepsilon \oslash \\ 
			\varepsilon ^{3}\oslash & \frac{1}{2}-\frac{1}{2}\varepsilon & 1-\frac{1}{2}\varepsilon +\varepsilon
			^{2}\oslash & | & \frac{1}{2}\varepsilon +\varepsilon \oslash%
		\end{array}%
		\right] \medskip \\
		&%
		\begin{tabular}{c}
			$L_{1}-L_{2}$ \\ 
			$\longrightarrow $ \\ 
			$L_{3}-\left( \frac{1-\varepsilon}{2} \right) L_{2}$%
		\end{tabular}%
		\left[ 
		\begin{array}{ccccccc}
			1+\varepsilon ^{2}\oslash  & \varepsilon ^{2}\oslash &  \varepsilon ^{3}\pounds  & |
			& -1+\varepsilon \oslash \\ 
			\varepsilon ^{2}\oslash  &1+\varepsilon ^{2}\oslash & 1+\varepsilon ^{3}\pounds  & |
			& 2+\varepsilon \oslash \\ 
			\varepsilon ^{2}\oslash & \varepsilon ^{2}\oslash & \frac{1}{2}+\varepsilon ^{2}\oslash
			& | & -1+\frac{3}{2}\varepsilon+\varepsilon \oslash%
		\end{array}%
		\right] 
			\end{align*}	
		\begin{align*}
			&%
		\begin{tabular}{c}
			$\longrightarrow $ \\ 
			$2L_{3}$ \\ 
		\end{tabular}%
		\left[ 
		\begin{array}{ccccc}
			1+\varepsilon ^{2}\oslash & \varepsilon ^{2}\oslash & \varepsilon ^{3}\pounds  & |
			& -1+\varepsilon \oslash \\ 
			\varepsilon ^{2}\oslash & 1+\varepsilon ^{2}\oslash & 1+\varepsilon ^{3}\pounds  & |
			& 2+\varepsilon \oslash \\ 
			\varepsilon ^{2}\oslash & \varepsilon ^{2}\oslash & 1+\varepsilon ^{2}\oslash
			& | & -2+3\varepsilon+\varepsilon \oslash%
		\end{array}%
		\right] \medskip 	\\	&	\begin{tabular}{c}
			$\longrightarrow $ \\ 
			$L_{2}-L_{3}$ \\ 
		\end{tabular}%
		\left[ 
		\begin{array}{ccccc}
			1+\varepsilon ^{2}\oslash  &   \varepsilon ^{2}\oslash   &\varepsilon ^{3}\pounds &|&   -1+\varepsilon \oslash \\ 
			\varepsilon ^{2}\oslash & 1+\varepsilon ^{2}\oslash & \varepsilon ^{2}\oslash  &|& 4-3\varepsilon + \varepsilon \oslash \\ 
			\varepsilon ^{2}\oslash & \varepsilon ^{2}\oslash & 1+\varepsilon ^{2}\oslash
			& | & -2+3\varepsilon+\varepsilon \oslash%
		\end{array}%
		\right] \equiv \mathcal{I}_3|S.
	\end{align*}	
	The system $\mathcal{I}_3|S  $ being stable, by Theorem \ref{maintheorem} the external vector $ S $ solves the latter system. It is easy to verify this by substitution, and it is straightforward to verify that Cramer's Rule yields the same solution. 
\end{example}

The next example deals with a system having a coefficient matrix with an infinitesimal determinant. Yet it is not an absorber of the neutrix occurring in the right-hand member. Also the remaining conditions for stability hold, so the Gauss-Jordan procedure still works. Again the solution will be an external vector with coordinates in the form of a truncated expansion, now starting with a "singular", i.e. unlimited term.
\begin{example}  Let $\eps$ be a positive infinitesimal. We will use the microhalos $ M_{\varepsilon}=\pounds\eps^{\not\infty} $ and $ M_{\varepsilon_{1}}=\pounds\eps_{1}^{\not\infty}  $, where $\eps_1=\eps^\omega$ with $\omega\in \N$ unlimited. 
	Consider the system $$\left\{ 
	\setlength\arraycolsep{0pt}
	\begin{array}{r >{{}}c<{{}} r  @{{}\subseteq {}} r  >{{}}c<{{}} r  >{{}}c<{{}} r }
		\left(1+\pounds \eps_1^{\not\infty}\right)x & + & y &  1& + & \pounds\eps^{\not\infty}\\
		x & + & (1-\eps+\pounds \eps^{\not\infty})y &   2& +& \pounds\eps^{\not\infty}
	\end{array}.
	\right.$$
	Let $\A=\begin{bmatrix}
		1+\pounds \eps_1^{\not\infty}& 1\\ 1& 1-\eps+\pounds \eps^{\not\infty}
	\end{bmatrix}$. Then $\Delta=\det\A=-\eps +\pounds \eps^{\not\infty}$ is zeroless. One easily verifies that the system is stable. Applying the Gauss-Jordan procedure we obtain 
	\begin{align*}
		\mathcal{A}|\mathcal{B} &= \begin{bmatrix}
			1+\pounds \eps_1^{\not\infty}& 1 &|& 1+\pounds\eps^{\not\infty}\\ 
			1& 1-\eps+\pounds \eps^{\not\infty}&|& 2+\pounds\eps^{\not\infty}
		\end{bmatrix} \\
		& 
		\begin{tabular}{c}
			$\longrightarrow $ \\ 
			$L_{2}-L_{1}$ \\ 
		\end{tabular}%
		\begin{bmatrix}
			1+\pounds \eps_1^{\not\infty}& 1 &|& 1+\pounds\eps^{\not\infty}\\ 
			\pounds \eps_1^{\not\infty} & -\eps+\pounds \eps^{\not\infty}&|& 1+\pounds\eps^{\not\infty}
		\end{bmatrix} \\
		&	\begin{tabular}{c}
			$\longrightarrow $ \\ 
			$-\frac{1}{\eps}L_{2}$ \\ 
		\end{tabular} \begin{bmatrix}
			1+\pounds \eps_1^{\not\infty}& 1 &|& 1+\pounds\eps^{\not\infty}\\ 
			\pounds \eps_1^{\not\infty} & 1+\pounds \eps^{\not\infty}&|& -\frac{1}{\eps}+\pounds\eps^{\not\infty}
		\end{bmatrix} \\
		&	\begin{tabular}{c}
			$\longrightarrow $ \\ 
			$L_1-L_{2}$ \\ 
		\end{tabular} \begin{bmatrix}
			1+\pounds \eps_1^{\not\infty}& \pounds \eps^{\not\infty} &|&\frac{1}{\eps}+1+ \pounds\eps^{\not\infty}\\ 
			\pounds \eps_1^{\not\infty} & 1+\pounds \eps^{\not\infty}&|& -\frac{1}{\eps}+\pounds\eps^{\not\infty}
		\end{bmatrix}.
	\end{align*} 
	By Theorem \ref{maintheorem} the vector $\xi=(\frac{1}{\eps}+1+ \pounds\eps^{\not\infty}, -\frac{1}{\eps}+\pounds\eps^{\not\infty})^T$ is the solution of the system. 
\end{example}

The following two examples show that flexible systems do not need to have exact solutions, and that the solution of a non-uniform system does not need to be an external vector; then it is also possible that the Gauss-Jordan operations lead to non-feasible solution, i.e. a vector which does not satisfy the original system.

\begin{example}\label{notexact}
	The simple  equation $ \oslash x\subseteq \pounds $ does not have an exact solution. Indeed, it is satisfied by all limited numbers, but not by any unlimited number. Hence $ S=\pounds $, with $ \oslash \pounds=\oslash\subset \pounds $.
\end{example}

\begin{example}\label{ex1change}
	Consider the flexible system  \begin{equation}\label{vd1change}	\left\{\begin{array}{rrrrl} 
			(1+\oslash)x &+& (1+\varepsilon\oslash)y & \subseteq & \oslash\\
			(1+\varepsilon\pounds)x & -& (1+\varepsilon\pounds) y  & \subseteq & \varepsilon\pounds
		\end{array}.\right. 
	\end{equation} 
	As shown in \cite{NamImme} the    solution  is given by 
	\begin{equation}\label{ntd}	
		N=\oslash
		\begin{pmatrix}
			\frac{1}{2}
			\\
			\frac{1}{2}\end{pmatrix} + \varepsilon\pounds \begin{pmatrix}
			\frac{1}{2}
			\\
			-\frac{1}{2}
		\end{pmatrix},
	\end{equation}
	which is not an external (neutricial) vector, though it is the result of applying a rotation to the neutricial vector $ (\oslash,\varepsilon\pounds) $.  
	
	To show that the Gauss-Jordan operations may not respect feasibility, we subtract the first equation from the second. Then we get 	
	\begin{equation*}\label{vd2change}	
		\left\{\begin{array}{rrrrl} 
			(1+\oslash)x &+& (1+\varepsilon\oslash)y & \subseteq & \oslash\\
			\oslash x & -& 2(1+\varepsilon\pounds) y  & \subseteq & \oslash
		\end{array}.\right. 
	\end{equation*} 
	The obvious solution is the neutricial vector $ K\equiv (\oslash,\oslash) $, but due to the fact that one neutrix at the right-hand side has been increased, it does no longer satisfy the original system. Indeed $ N\subset K $, for instance the representative vector $ (0,\sqrt{\varepsilon}) $ does not satisfy the second equation of \eqref{vd1change}.
\end{example}

We now turn to the stability conditions.

Example \ref{ex1} shows that, without the condition stating that the relative uncertainty of the coefficient matrix must be smaller than the relative uncertainty of the constant term, a non-singular flexible system may have no solution at all. 
\begin{example} \label{ex1} Let $ \varepsilon\simeq 0,\varepsilon\neq 0 $.  The equation $ (1+\oslash)x\subseteq 1+\varepsilon\pounds $ has no solution. 
\end{example}
The next example shows that if the determinant of the coefficient matrix is an absorber of  the neutrix of the right-hand side, the solution given by the Gauss-Jordan procedure may be not feasible.
\begin{example}\label{notab}  Consider the  system 
	$$\left\{\begin{array}{rrrrrrrrll}
		x_1& +&x_2   &\subseteq & 1& +& \oslash\\
		& & \eps x_2 & \subseteq&   \oslash
	\end{array},\right.$$
	with $ \varepsilon \simeq 0,\varepsilon\neq 0$. The determinant of the coefficient matrix  $\Delta=\eps$ is an absorber of the neutrix of the right-hand side $B=\oslash$. Applying Gauss-Jordan elimination we blow $ B $ up  to $ \oslash/\eps $, and obtain at the right-hand side the neutrix vector $\left(\oslash/\eps, \oslash/\eps\right)^T$, which is obviously not admissible.
\end{example}
The stability conditions are stated in terms of bounds, and as may be expected, they are not minimal. This is illustrated by the final example.

\begin{example}\label{notmin} Consider the following  system 
	$$\left\{\begin{array}{rrrrrrrrll}
		x_1& &   &\subseteq & 1& +& \oslash\\
		& & \eps x_2 & \subseteq&   \oslash
	\end{array},\right.$$
	with $ \varepsilon \simeq 0,\varepsilon\neq 0$. As in Example \ref{notab} the determinant $\Delta=\varepsilon$ is an absorber of $B=\oslash$. Gauss-Jordan elimination only consists in multiplying the second inclusion by $ 1/\varepsilon $, and leads to $\G^P(B)=\left(1+\oslash, \oslash/\eps\right)^T$, which is the solution of the system indeed. 
\end{example}
\section{Preliminary results}\label{preliminary}
In Subsection \ref{excalculus} we recall some useful properties of the calculus with external numbers and matrices, and  Subsection \ref{exex} contains explicit expressions for entries of intermediate matrices of the Gauss-Jordan elimination procedure and the Gauss-Jordan operation matrices.

\subsection{On the calculus of external numbers and matrices}\label{excalculus}

We will consider the modified distributive law for external numbers, some additional properties, and properties of matrix multiplication, in particular modified laws for distributivity and also associativity. We end with some properties concerning the order of magnitude of determinants and minors.

The distributive law holds for the external numbers under fairly general conditions, but in particular it may not hold when multiplying two almost opposite numbers. Unfortunately, this is common practice in the context Gauss-Jordan operations, for we search for zero's or neutrices by annihilating. However subsdistributivity always holds, and this does not affect the inclusions we work with. 

\begin{theorem}{\rm \cite{Dinis}}{(Distributivity with correction term)}\label{tcopti2} Let $\alpha, \beta, \gamma=c+C$ be external numbers. Then 
	\begin{equation}\label{dis}
		\alpha \gamma+\beta \gamma=(\alpha+\beta)\gamma+C \alpha+C \beta.
	\end{equation}  
\end{theorem}
Because a neutrix term is added in the right-hand side of \eqref{dis}, we always have the following form of subdistributivity.

\begin{corollary}{(Subdistributivity)}\label{tcopti1}  
	Let $\alpha, \beta, \gamma$ be external numbers. Then $(\alpha+\beta)\gamma\subseteq\alpha \gamma+\beta \gamma$.
\end{corollary}

Theorem \ref{tcopti3} below gives conditions such that the common distributive law holds, i.e. the correction terms figuring in \eqref{dis} may be neglected. To this end we recall the notions of \emph{relative uncertainty} and \emph{oppositeness}. 
\begin{definition}{\rm \cite{Koudjeti Van den Berg, Dinis}}\label{dnbruno}
	Let $\alpha=a+A$ and $  \beta=b+B $ be external numbers and $ C $ be a neutrix.
	\begin{enumerate}
		\item \label{ru}The \emph{relative uncertainty}  $ R(\alpha) $  of $\alpha$ is defined by $ R(\alpha)=A/\alpha $ if $\alpha $ is zeroless, otherwise $ R(\alpha)=\R $.
		\item\label{opp} $\alpha$ and $  \beta$ are \emph{opposite} with respect to $ C $ if $ (\alpha+\beta)C\subset \max (\alpha C, \beta C). $
	\end{enumerate}
\end{definition}

\begin{theorem}\label{tcopti3}Let $\alpha, \beta, \gamma=c+C$ be external numbers. Then $\alpha \gamma+\beta \gamma=(\alpha+\beta)\gamma$ if and only if    
	$R(\gamma)\subseteq \max( R(\alpha), R(\beta))$, or  $\alpha $ and  $\beta $ are not opposite with respect to $ C $. 
\end{theorem}

Simple and important special cases are given by
\begin{equation*}
	(x+N)\beta=x\beta+N\beta \mbox{ and  } x(\alpha+\beta)=x\alpha+x\beta,
\end{equation*}
whenever  $x\in \mathbb{R}$, $N\in \mathcal{N}$ and $ \alpha,\beta \in \E$.

Next proposition lists some useful general properties of external numbers. 
\begin{proposition}\label{tcopt}\cite{Koudjeti Van den Berg} 
	Let $\alpha=a+A$ and $ \gamma $ be a zeroless external numbers, $ B$ be a neutrix and  $n\in \mathbb{N}$ be standard. Then 
	\begin{enumerate}
		
		\item \label{tcoptiv}$\alpha B=aB$ and $	\frac{B}{\alpha}=\frac{B}{a}$.
		\item \label{tcoptid}$N(1/\alpha)=N(\alpha)/\alpha^{2} $.
		\item \label{tcoptix}  $R(\alpha),R(1/\alpha)\subseteq \oslash$. 	
		\item \label{giao} $\alpha		\cap \oslash \alpha=\emptyset$.	
		\item \label{tcoptp} $ N(\alpha\gamma )=\alpha N(\gamma )+ N(\alpha)\gamma $.	
		\item \label{tcoptv}  $N\big((a+A)^n\big)=a^{n-1}A.$		
		\item \label{lei3} If $\alpha $ is limited and is not an absorber of $B$, then $\alpha B=\frac{B}{\alpha}=B.$
	\end{enumerate}
\end{proposition}
Below we give a brief account of some relevant properties of matrices over external numbers. We refer to \cite{Imme nam 3} for more details, proofs and examples.  

The following general property of inclusion is an immediate consequence of the fact that, given external numbers $\alpha, \beta, \gamma$ such that $\alpha\subseteq \beta$, one has $\gamma\alpha\subseteq \gamma \beta$. 
\begin{proposition}\label{matincl}
	Let $\mathcal{A}\in \M_{m, n}(\E)$ and $ \mathcal{B}, \mathcal{C}\in \M_{n, p}(\E)$. If $ \mathcal{B}\subseteq \mathcal{C}$ then $ \mathcal{A}\mathcal{B}\subseteq\mathcal{A}\mathcal{C} $.
\end{proposition}
Because subdistributivity holds for external numbers, it  also holds for the calculus of matrices of external numbers. Next proposition gives a condition for distributivity.
\begin{proposition}\label{distributivitymatrix} Let $\A=[\alpha_{ij}]_{m\times n}\in \M_{m,n}(\E)$ and $ \B=[\beta_{ij}]_{n\times p}, \C=[\gamma_{ij}]_{n\times p}\in \M_{n,p}(\E)$. If $\max\limits_{\substack{1\leq i\leq m\\1\leq j\leq n}} R(\alpha_{ij})\leq \min\limits_{\substack{1\leq i\leq m\\1\leq j\leq n}}\max \{R(\beta_{ij}), R(\gamma_{ij})\}$,  then  $$\A(\B+\C)= \A\B+\A\C.$$	
\end{proposition}
For subassociativity to hold conditions are needed, and associativity holds under stronger conditions.
\begin{proposition}\label{subassociativity}
	Let $\mathcal{A}\in \M_{m, n}(\E), \mathcal{B}\in \M_{n, p}(\E)$ and $ \mathcal{C}\in \M_{p, q}(\E)$. Then
	\begin{enumerate}
		\item \label{inclusion1}  $(\mathcal{A}\B)\mathcal{C}\subseteq \mathcal{A}(\mathcal{B}\mathcal{C}) \mbox{ if $\mathcal{A}$ is a real matrix or $\mathcal{B}, \mathcal{C}$ are both non-negative}.$
		\item  \label{inclusion2}  $\mathcal{A}(\mathcal{B}\mathcal{C})\subseteq (\mathcal{A}\mathcal{B})\mathcal{C} \mbox{ if $\mathcal{C}$ is a real matrix or $\mathcal{A}, \mathcal{B}$ are both non-negative}.$
	\end{enumerate}
\end{proposition}
\begin{proposition} \label{combinetheorem}
	Let $\mathcal{A}\in \M_{m, n}(\E), \mathcal{B}\in \M_{n, p}(\E)$ and $ \mathcal{C}\in \M_{p, q}(\E)$. Then $\mathcal{A}(\mathcal{B}\mathcal{C})=(\mathcal{A}\mathcal{B})\mathcal{C}$ if one of the following conditions is satisfied:
	\begin{enumerate}
		\item \label{associate of product of matrix} $\mathcal{A}$ and $\mathcal{C}$ are both real matrices.
		\item \label{associate of product of matrix2}  $\mathcal{B}$ is a neutricial matrix.
		\item \label{nonnegass} $\mathcal{A},\mathcal{B},\mathcal{C}$ are all non-negative matrices.
	\end{enumerate}
\end{proposition}

In the final part we relate some orders of magnitude of determinants of limited and reduced matrices and its minors.

To start with, it is easily proved that the determinant of a limited matrix is limited, as are its minors. The neutrix of these determinants does not exceed the biggest neutrix of the entries.	

\begin{proposition}
	\label{cdt} Let $n\in \mathbb{N}$ be standard and $\mathcal{A}%
	\in \mathcal{M}_{n}(\mathbb{E})$ be limited. Then there exists a limited number $ L>0 $ such that whenever 
	$k\in \{1,\dots, n\}$ and $1\leq i_1<\dots<i_k\leq n,\ 1\leq
	j_1<\dots<j_k\leq n$
	\begin{equation*}
		|m^{i_1\dots i_k}_{j_1\dots j_k}|\leq L.
	\end{equation*}
	In particular $ |\Delta|\leq L$. Moreover $N(\Delta)\subseteq \oA.$
\end{proposition}

The last property plays an important part in our approach to error analyis, and says that at least one the minors $\Delta_{i,j}$, obtained by eliminating row $ i $ and column $ j $ from the matrix $ \A $ for some $i, j$ with $ 1\leq i,j\leq n $, is of the same order of magnitude as the determinant. It is a consquence of the fact that the Laplace expansion holds with inclusions. 
\begin{proposition}{\rm }\label{lei2} Let $\A\in \M_{n}(\E)$ be a  reduced square matrix of order $n$. Suppose that $\Delta$ is zeroless. Then for each $j\in \{1, \dots, n\}$ there exists  $i\in \{1, \dots,n\}$ such that $|\Delta_{i,j}|>\oslash\Delta$.
\end{proposition}

\subsection{Explicit expressions for the Gauss-Jordan operations}\label{exex}

We will use explicit expressions for the Gauss-Jordan operation matrices and the intermediate matrices. These are given in terms of quotients of minors, for which we recall the convenient notation of \cite{Gantmacher}. Proofs can be found in \cite{Li}, in a different notation, and in \cite{NJI1}. In particular a pivot is always given in the form of a quotient of principal minors, of which the order of magnitude can be determined with the methods of Subsection \ref{excalculus}. At the end we consider the inverse Gauss-Jordan procedure.

\begin{theorem}[Explicit expressions for Gauss-Jordan operations]\label{Gmat}Let \linebreak $\A=[a_{ij}]_{n\times n}\in \M_{n}(\mathbb{R}) $ be  diagonally eliminable. For $ k<n $ the Gaussian elimination matrix of odd order
	$ 	\mathcal{G}_{2k+1}=[	g_{ij}^{(2k+1)}]_{n\times n}$ 	satisfies
	\begin{equation}\label{Gmatodd}
		g_{ij}^{(2k+1)}=%
		\begin{cases}
			\quad 1 & \mbox{ if }i=j\not=k+1 \\ 
			\quad 0 & \mbox{ if }i\not=j \\ 
			\dfrac{m_{k}}{m_{k+1}} & \mbox{ if }i=j=k+1%
		\end{cases}%
	\end{equation}
	and the Gaussian elimination matrix of even order
	$
	\begin{array}{rr}
		\mathcal{G}_{2k+2}=%
		\begin{bmatrix}
			g_{ij}^{(2k+2)}%
		\end{bmatrix}%
		_{n\times n} 
	\end{array}%
	$ satisfies
	\begin{equation*}
		g_{ij}^{(2k+2)}=%
		\begin{cases}\qquad \qquad \qquad 0 & \mbox{ if }j\not \in \{i,k+1\} \\ 
			\qquad \qquad \qquad 1 & \mbox{ if }j=i \\ 
			(-1)^{k+i+1} \dfrac{m^{1\dots k}_{1\dots i-1i+1\dots k+1}}{m_{k}}%
			\medskip  & \mbox{ if }1\leq i\leq k,j= k+1 \\ 
			\qquad \qquad -\dfrac{m^{1\dots ki}_{1\dots kj}}{m_{k}} & \mbox{ if }k+1<i \leq n,j= k+1%
		\end{cases}.%
	\end{equation*}
	
\end{theorem}

\begin{theorem}[Explicit expressions for Gauss-Jordan elimination]
	\label{cmcttruyhoi} Let 
	$\A= [a_{ij}]_{n\times n}\in \M_{n}(\mathbb{R})$ be   diagonally eliminable.  Let $k<n$. Then 
	\begin{equation*}
		\A^{(2k)}=%
		\begin{bmatrix}
			1 & \cdots  & 0 & a_{1k+1}^{(2k)} & \cdots  & a_{1n}^{(2k)} \\ 
			\vdots  & \ddots  & \vdots  & \vdots  & \ddots  & \vdots  \\ 
			0 & \cdots  & 1 & a_{kk+1}^{(2k)} & \cdots  & a_{kn}^{(2k)} \\ 
			0 & \cdots  & 0 & a_{k+1k+1}^{(2k)} & \cdots  & a_{k+1n}^{(2k)} \\ 
			\vdots  & \ddots  & \vdots  & \vdots  & \ddots  & \vdots  \\ 
			0 & \cdots  & 0 & a_{nk+1}^{(2k)} & \cdots  & a_{nn}^{(2k)}%
		\end{bmatrix},%
	\end{equation*}%
	where 
	\begin{equation}
		a_{ij}^{(2k)}=%
		\begin{cases}
			(-1)^{k+i} \dfrac{m^{1\dots k}_{1\dots i-1i+1\dots kj}}{m_{k}}%
			\medskip  & \mbox{ if }1\leq i\leq k,k+1\leq j\leq n \\ 
			
			\qquad \qquad \dfrac{m^{1\dots ki}_{1\dots kj}}{m_{k}} & \mbox{ if }k+1\leq i\leq n, k+1\leq j\leq n%
		\end{cases}.%
		\label{congthucquynapcuaA}
	\end{equation}%
	In particular $ \A^{(2n)}=I_n $.
\end{theorem}

If $ \A $ is diagonally eliminable, the inverses of the matrices of the Gauss-Jordan procedure $ \G  $ are well-defined, as follows. For odd indices we have $\G^{-1}_{2k+1}=\left[\left(g^{-1}_{ij}\right)^{(2k+1)}\right]$, with 	\begin{equation*}
	\left(	g^{-1}_{ij}\right)^{(2k+1)}=%
	\begin{cases}
		\quad 1 & \mbox{ if }i=j\not=k+1 \\ 
		\quad 0 & \mbox{ if }i\not=j \\ 
		\dfrac{m_{k+1}}{m_{k}} & \mbox{ if }i=j=k+1%
	\end{cases},%
\end{equation*}
and for even indices
$\begin{array}{rr}
	\mathcal{G}^{-1}_{2k+2}=%
	\begin{bmatrix}
		\left(g^{-1}_{ij}\right)^{(2k+2)}%
	\end{bmatrix}%
	_{n\times n} 
\end{array}%
$, with
\begin{equation*}
	\left(g^{-1}_{ij}\right)^{(2k+2)}=%
	\begin{cases}\qquad \qquad \qquad 0 & \mbox{ if }j\not \in \{i,k+1\} \\ 
		\qquad \qquad \qquad 1 & \mbox{ if }j=i \\ 
		(-1)^{k+i} \dfrac{m^{1\dots k}_{1\dots i-1i+1\dots k+1}}{m_{k}}%
		\medskip  & \mbox{ if }1\leq i\leq k,k+1\leq j\leq n \\ 
		\qquad \qquad \dfrac{m^{1\dots ki}_{1\dots kj}}{m_{k}} & \mbox{ if }k+1<i\leq n, k+1\leq j\leq n%
	\end{cases}.%
\end{equation*}

The sequence  $\A^{(2n)}, \G^{-1}_1\left(\A^{(2n)}\right),\dots, \mathcal{G}^{-1}_{1}\left(\cdots\left(\mathcal{G}^{-1}_{2n}\A^{(2n)}\right)\right) =\A$ is called the \emph{inverse Gauss-Jordan procedure}.

\section{Stability and Gauss-Jordan operations}\label{lemmas}
For reduced and properly arranged matrices, when applying the Gauss-Jordan operations to the intermediate matrices, only a moderate growth is possible for the elements and their neutrix parts. If the determinant is not an absorber of the neutrix part of the right-hand member, this neutrix even remains constant. If in addition the flexible system is stable, a stable matrix is transformed into a stable matrix, while the relative uncertainty of the intermediate matrices remains always less than the relative uncertainty of the right-hand members. Together this leads to a proof of Theorem \ref{lmqhnt} on the preservation of stability under the Gauss-Jordan operations, with the final matrix being a near-identity matrix. 

The principal tools in establishing the above properties of orders of magnitude and stability are bounds on the order of magnitude of minors. Indeed, because the pivots are quotients of minors, they have direct influence on the order of magnitude of the entries and neutrix parts of the intermediate matrices and the right-hand members.

\begin{remark}
	We recall from the previous section that a reduced matrix $ \A $ has always a reduced representative matrix, and from now on we always suppose that a representative matrix is reduced.
\end{remark}
Proposition \ref{limitofelement} shows that the Gauss-Jordan operations do not lead to an unlimited growth of the elements of the intermediate matrices. 
\begin{proposition}
	\label{limitofelement} Let $\mathcal{A}=[\alpha_{ij}]_{n\times n}\in \M_{n}(\E)$ be a reduced  non-singular matrix, such that it admits a  properly arranged  representative matrix $ P $. Then $a^{(q)}_{ij}$ is limited whenever $1\leq q\leq 2n$ and $%
	1\leq i, j\leq n$.
\end{proposition}

\begin{proof}
	We apply external induction. Because  $P$ is reduced, it holds that $|a_{ij}|\leq 1$ for  $1\leq i,j\leq n$ and, since $a_{ij}^{(1)}=a_{ij}$, the same is true for $\left|a_{ij}^{(1)}\right|$. It follows that $\left|a^{(2)}_{1j}%
	\right|= \left|a_{1j}\right|\leq 1$ for $1\leq j\leq n$ and $\left|a^{(2)}_{ij}\right|=%
	\left|a_{ij}-a_{i1} a_{1j}\right|\leq
	\left|a_{ij}\right|+\left|a_{i1}\right|\left|a_{1j}\right|\leq 2$ for
	$2\leq i\leq n, 1\leq j\leq n.$  Hence $%
	a^{(2)}_{ij}$ is limited for $1\leq i, j\leq n$. 
	As for the induction step, let $ k\leq n-1 $ and suppose that $	a^{(q)}_{ij}$ is limited for $q\leq 2k $ and $1\leq i, j\leq n$.
	Because the $j^{th}$ column of $P^{(2k+1)}$ is a unit vector for $1\leq
	j\leq k$, the entries of these columns are limited. For $1\leq i\leq n, k+1\leq j\leq n$ one has 
	\begin{equation*}
		a^{(2k+1)}_{ij}=%
		\begin{cases}
			a^{(2k)}_{ij} & \text{ if } i\not =k+1 \\ 
			\dfrac{m^{1...ki}_{1...kj}}{m_{k+1}} & \text { if } i=k+1%
		\end{cases}.%
	\end{equation*}
	So $a^{(2k+1)}_{ij}=a^{(2k)}_{ij}$ is limited for $i\not =k+1$ and $ k+1\leq j\leq n$ by the induction hypothesis, and because $ P$ is properly arranged, also for $i=k+1$ and $k+1\leq j\leq n$, since 
	$\left|a^{(2k+1)}_{k+1j} \right|\leq\left|\dfrac{m^{1...ki}_{1...kj}}{m_{k+1}} \right| \leq 1$. Combining, we see that $a^{(2k+1)}_{ij}$ is limited for $%
	1\leq i, j\leq n$.
	
	As for $ P^{(2k+2)} $, in addition to the first $ k $ columns, also the   $(k+1)^{th}$ column is a unit vector, i.e. has limited components. Because the elements of $ P^{(2k+1)} $ are limited we derive that  $a^{(2k+2)}_{k+1j}=a^{(2k+1)}_{k+1j}$
	is limited for $k+2\leq j\leq n$, and   $a^{(2k+2)}_{ij}=a^{(2k+1)}_{ij}-a^{(2k+1)}_{ik+1}%
	a^{(2k+1)}_{k+1j}$ is limited for $ 1\leq i\leq n,i\not =k+1 $ and $k+1\leq j\leq n$.  Hence $a^{(2k+2)}_{ij}$ is limited for $%
	1\leq i, j\leq n$.
\end{proof}

The quotients of the principal minors and \emph{a fortiori} the pivots have definite lower bounds and upper bounds in terms of neutrices. This is shown in Theorem \ref{propmm}. This theorem includes bounds for the pivots of the inverse procedure, which we will need to verify that the Gauss-Jordan solution is a solution of the original system. To prove the theorem we present first some notation and an auxiliary result, saying that the determinants of the intermediate matrices are at least of the same order of magnitude as the determinant of the original matrix.

\begin{notation}
	Let $\mathcal{A}=[\alpha_{ij}]_{n\times n}\in \M_{n}(\E)$ be a reduced non-singular  matrix, such that it admits a properly arranged representative matrix $ P=[a_{ij}]_{n\times n} $. For $ 1\leq q\leq 2n $ we write $ d=\det(P) $, $  d^{(q)}=\det(P^{(q)})$ and $\Delta^{(q)}= \det\A^{(q)}=d^{(q)}+D^{(q)}$. 
\end{notation}
\begin{lemma}\label{lower bound of d2k} 
	Let $\mathcal{A}=[\alpha_{ij}]_{n\times n}\in \M_{n}(\E)$ be a  reduced non-singular matrix, such that it admits a properly arranged representative matrix $ P $. Let $1\leq q\leq 2n$  and $ k $ be such that $q=2k-1$ or $q=2k$. Then $\left|d^{(q)}\right|= \left|\dfrac{d}{m_{k}}\right|>\oslash \Delta $.
\end{lemma}
\begin{proof} Let $1\leq k\leq n$ and $q=2k-1$ or $q=2k$. In both cases	
	\begin{alignat}{2}
		\left|d^{(q)}\right|=&\left|\det(\mathcal{G}_{q}) \det(\mathcal{G}%
		_{q-1})\cdots \det(\mathcal{G}_{1}) d\right| & \label{formula of d2k new} \\
		=&\left|\dfrac{m_{k-1}}{m_{k}}\dfrac{m_{k-2}}{m_{k-1}}\cdots \dfrac{%
			m_{1}}{m_{2}}\dfrac{1}{m_{1}}d\right|  = \left|\dfrac{d}{m_{k}}\right|. &\notag  
	\end{alignat}
	Suppose $|d^{(q)}|\subseteq \oslash \Delta.$ Then $d\in m_{k}\oslash \Delta$ by \eqref{formula of d2k new}.
	By Proposition \ref{cdt} it holds that $d\in \oslash \Delta$. Hence $d\in \oslash \Delta \cap \Delta$. Because $\Delta$ is zeroless, this contradicts  Proposition \ref{tcopt}.\ref{giao}. Hence $\left|d^{(q)}\right|>\oslash \Delta$.		
\end{proof}

\begin{theorem}\label{propmm}
	Let $\mathcal{A}=[\alpha_{ij}]_{n\times n}\in \M_{n}(\E)$ be a reduced, non-singular matrix, which is properly arranged with respect to a   matrix of  representatives $P$.    Then for $ 1\leq k< n $
	\begin{equation}\label{mk+1mk}
		\oslash\Delta <%
		\begin{vmatrix}
			\dfrac{m_{k+1}}{m_{k}}%
		\end{vmatrix}
		\in \pounds
	\end{equation}
	and 
	\begin{equation}\label{mkmk+1}
		\oslash <%
		\begin{vmatrix}
			\dfrac{m_{k}}{m_{k+1}}%
		\end{vmatrix}
		\in \frac{\pounds}{\Delta}.
	\end{equation}	
\end{theorem}
\begin{proof}
	For $1\leq k\leq n-1$ we have 
	\begin{alignat*}{2}
		\mathcal{A}^{(2k)}= & 
		\begin{bmatrix}
			1+A_{11} & A^{(2k)}_{12} & \cdots & A^{(2k)}_{1k} & 
			\alpha^{(2k)}_{1(k+1)} & \cdots & \alpha^{(2k)}_{1n} \\ 
			A^{(2k)}_{21} & 1+A^{(2k)}_{22} & \cdots & A^{(2k)}_{2k} & 
			\alpha^{(2k)}_{2(k+1)} & \cdots & \alpha^{(2k)}_{2n} \\ 
			\vdots & \vdots & \ddots & \vdots & \vdots & \ddots & \vdots \\ 
			A^{(2k)}_{k1} & A^{(2k)}_{k2} & \cdots & 1+A^{(2k)}_{kk} & 
			\alpha^{(2k)}_{k(k+1)} & \cdots & \alpha^{(2k)}_{kn} \\ 
			A^{(2k)}_{(k+1)1} & A^{(2k)}_{(k+1)2} & \cdots & A^{(2k)}_{(k+1)k} & 
			\alpha^{(2k)}_{(k+1)(k+1)} & \cdots & \alpha^{(2k)}_{(k+1)n} \\ 
			\vdots & \vdots & \ddots & \vdots & \vdots & \ddots & \vdots \\ 
			A^{(2k)}_{n1} & A^{(2k)}_{n2} & \cdots & A^{(2k)}_{nk} & 
			\alpha^{(2k)}_{n(k+1)} & \cdots & \alpha^{(2k)}_{nn}%
		\end{bmatrix}.
	\end{alignat*}
	Suppose on the contrary that $\dfrac{m_{k+1}}{m_{k}}=a^{(2k)}_{k+1k+1}=\dfrac{m_{k+1}}{m_{k}}%
	\in \oslash \Delta$. From $\left|a^{(2k)}_{ij}\right|=\left|\dfrac{%
		m^{1...ki}_{1...kj}}{m_{k}}\right|\leq \left|\dfrac{m_{k+1}}{m_{k}}%
	\right|=\left|a^{(2k)}_{k+1k+1}\right|$ for $k+1\leq i, j\leq n$
	one derives that $a^{(2k)}_{ij}\in \oslash \Delta$ for  $k+1\leq i,
	j\leq n$. Let $S_{n-k}$ be the set of all permutations of $\{k+1,\dots, n\}$. Because $ \Delta $ is limited,
	\begin{equation*}  \label{danh gia d2k}
		d^{(2k)}=\displaystyle \sum_{\sigma \in S_{n-k}}\sgn%
		(\sigma)a^{(2k)}_{k+1\sigma(k+1)} \dots a^{(2k)}_{n\sigma(n)}\in
		(\oslash \Delta)^{n-k}\subseteq \oslash \Delta,
	\end{equation*}
	in contradiction to Lemma \ref{lower bound of d2k}. Hence $\left|\dfrac{m_{k+1}}{m_{k}}\right|>\oslash \Delta$.
	
	Also, by Theorem \ref{Gmat} we have $m_{k+1}/m_{k}=a^{(2k+2)}_{k+1k+1}$, and the latter is limited by Proposition \ref{limitofelement}.	Hence formula \eqref{mk+1mk} holds.  Taking multiplicative inverses, we derive \eqref{mkmk+1}.
\end{proof}

Theorem \ref{propmm} gives bounds on the pivots and the entries of the elementary matrices of the Gauss-Jordan procedure, and the inverse procedure.

\begin{theorem}
	\label{tinhgioihancuasohang} Let $\mathcal{A}=[\alpha_{ij}]_{n\times n}\in \M_{n}(\E)$ be  a reduced, non-singular  matrix, which is properly arranged with respect to a  matrix of representatives $P$.
	\begin{enumerate}
		\item\label{cov} 
		Let $ 1\leq k<n $. Then the $ k^{th} $ diagonal element  of $\mathcal{G}^{P}_{2k+1}$ satisfies $ g^{(2k+1)}_{k+1k+1}\in \frac{\pounds}{\Delta} $ and the elements of $ \G^{P}_{2k+2}$ are all limited.
		\item\label{contrav} All elements of the  matrices $ (\G^{-1})^{P}_{q}, 1\leq q\leq 2n$ of the inverse Gauss-Jordan procedure are limited.
	\end{enumerate}
\end{theorem}
\begin{proof}
	
	\begin{enumerate}		
		\item The property is a direct consequence of Theorem \ref{propmm} and Proposition \ref{limitofelement}.		
		\item For the intermediate matrices of odd index the property follows from \eqref{mk+1mk}, and for the intermediate matrices of even index $ q=2k, k<n $ the property follows from the fact that $\left| \left( g^{-1}\right)^{(2k+2)}_{ik+1}\right| =\left| g^{(2k+2)}_{ik+1}\right| $ for $1\leq i\leq n, 1\leq k\leq n-1$, and Part \ref{cov}.		
	\end{enumerate}	
\end{proof}

With the help of Theorem \ref{tinhgioihancuasohang} we derive a bound for the possible increase of the neutrix parts of the intermediate matrices of the Gauss-Jordan procedure. If in addition the original matrix is stable, Lemma \ref{lower bound of d2k} permits to prove that they always are infinitesimal, implying that the intermediate matrices remain both non-singular and stable, until obtaining a near-identity matrix at the end.


\begin{proposition}\label{value of Ak}
	Let $\mathcal{A}\in \M_{n}(\E)$ be  a reduced, non-singular matrix, which is properly arranged with respect to a matrix of representatives $P$. Then for all  $ k $ such that $1\leq k\leq n$, 
	\begin{equation*}
		\overline{A^{(2k)}}=\overline{A^{(2k-1)}}\subseteq \dfrac{\oA}{m_k}.
	\end{equation*} 
\end{proposition}
\begin{proof}
	We will apply external induction. For $k=1$, because $m_1=a_{11}=1$, one has $\overline{A^{(2k-1)}}=\overline{A^{(1)}}=\oA=\dfrac{\oA}{m_1}$. By Part \ref{cov}  of Theorem \ref{tinhgioihancuasohang}  it holds that $g^{(2k)}_{ij}=g^{(2)}_{ij}$ is limited for $1\leq i, j\leq n$, hence $\overline{A^{(2)}}=\overline{A^{(1)}}=\dfrac{\oA}{m_1}$. 
	
	As for the induction step, let $k<n$ and suppose that $\overline{A^{(2k-1)}}=\overline{A^{(2k)}}\subseteq \dfrac{\oA}{m_k}$. Then $\overline{A^{(2k+1)}}\subseteq \dfrac{m_k}{m_{k+1}}\overline{A^{(2k)}}\subseteq \dfrac{m_k}{m_{k+1}} \dfrac{\oA}{m_k}=\dfrac{\oA}{m_{k+1}}$. Again, by Part \ref{cov} of Theorem \ref{tinhgioihancuasohang} one has $\overline{A^{(2k+2)}}=\overline{A^{(2k+1)}}=\dfrac{\oA}{m_{k+1}}.$ 
\end{proof} 

\begin{proposition}\label{zeroless of Deltak}
	Let $\mathcal{A}=[\alpha_{ij}]_{n\times n}\in \M_{n}(\E)$ be  a reduced, non-singular stable matrix, which is properly arranged with respect to a matrix of representatives $P$. Let  $1\leq q\leq 2n$.
	Then 
	\begin{enumerate}
		\item \label{kzeroless} $\Delta^{(q)}$ is zeroless.
		\item \label{klimited}$\oslash\Delta<\Delta^{(q)}\subset \pounds$. 
		\item \label{kmax}$ \overline{A^{(q)}}\subseteq \oslash\Delta^{(q)}\subseteq \oslash $.
		\item \label{kstable}$ \A^{(q)} $ is limited, non-singular and stable.
	\end{enumerate}
\end{proposition}
\begin{proof}	
	\ref{kzeroless}.
	Let  $q=2k$ or $q=2k-1$ with $1\leq k\leq n$. By Lemma \ref{lower bound of d2k} one has $\left|d^{(q)}\right|=\left|\dfrac{d}{m_k}\right|$. Because the matrix is non-singular and stable, it holds that $\overline{A}\subseteq \oslash\Delta <|d|$, and because it is also reduced, it follows from Proposition \ref{value of Ak} that $\overline{A^{(q)}}\subseteq\dfrac{\oA}{m_k}$. Hence $\left|d^{(q)}\right|>\overline{A^{(q)}}$. Also  $D^{(q)}\subseteq\overline{A^{(q)}}$ by Proposition \ref{limitofelement} and Proposition \ref{cdt}. Hence $\Delta^{(q)}$ is zeroless. 
	
	\ref{klimited}.  We show first that $D^{(q)}\subseteq \oslash$. Indeed, suppose $\oslash\subset D^{(q)}$. Then $\pounds \subseteq D^{(q)}$. By Proposition \ref{cdt} and Proposition \ref{limitofelement} it holds that $d^{(q)}$ is limited. This implies that $\Delta^{(q)}$ is a neutrix, in contradiction to Part \ref{kzeroless}. Hence $D^{(q)}\subseteq \oslash$, which implies that $\Delta^{(q)}=d^{(q)}+D^{(q)}\subset \pounds$. It also follows from Part \ref{kzeroless} that $ \Delta^{(q)}\subseteq (1+\oslash)d^{(q)} $. Now $ d^{(q)}>\oslash \Delta $ by Lemma \ref{lower bound of d2k}, hence also $ \Delta^{(q)}>\oslash \Delta $.
	
	\ref{kmax}. Let $1\leq q\leq 2n$. Then $q=2k$ or $q=2k-1$ with $1\leq k\leq n$. By Proposition \ref{value of Ak}, the stability of the matrix $ \A $,    Lemma \ref{lower bound of d2k} and Part  \ref{klimited}, one has 	
	\begin{equation*}\label{AkDk}
		\overline{A^{(q)}}\subseteq \dfrac{\oA}{m_k}\subseteq \dfrac{\oslash d}{m_k}=\oslash d^{(q)}=\oslash \Delta^{(q)} \subseteq\oslash.
	\end{equation*}	

	\ref{kstable}. By Proposition \ref{limitofelement} the matrix $A^{(q)}$ is limited. By Part \ref{kzeroless} the matrix $\A^{(q)}$ is non-singular. Then $\A^{(q)}$ is stable by Part \ref{kmax}. 
\end{proof}

\begin{theorem}\label{thenearid} 	
	Let $\mathcal{A}\in \M_{n}(\E)$ be  a reduced, non-singular stable matrix, which is properly arranged with respect to a  matrix of representatives $P$. Then $\G^{P}(\A)$ is a near-identity matrix.
\end{theorem}

\begin{proof}
	Let $A$ be the associated neutricial matrix of $\A$. By Proposition \ref{distributivitymatrix} we have $\G^P(\mathcal{A})=\mathcal{G}^P(P)+%
	\G^P(A)=I+A^{\prime }$, where $A^{\prime }=[A^{\prime }_{ij}]_{n\times n}$ is
	a neutricial matrix.  By Part \ref{kmax} of Proposition \ref{zeroless of Deltak} one has $A^{\prime }\subseteq [\oslash]_{n\times n}$. Hence $\G^{P}(\A)$ is a near-identity matrix.
\end{proof} 
We consider now the effect of the Gauss-Jordan procedure on the right-hand member of the system $ \A|\B $, where we always assume that the system satisfies Convention \ref{conv}. In fact, in contrast to the the possible increase of the neutrix parts of the coefficient matrix of a stable system, the pivots of the Gauss-Jordan procedure do not change the neutrix part of the right-hand member, and the same is true for the inverse procedure. The invariance of the neutrix part will be a consequence of the next proposition. 

\begin{proposition}
	\label{tisogiunguyen} Suppose that the flexible system $ \A|\B $ is properly arranged with respect to a representative matrix $P$. Assume that $ \Delta $ is not an absorber of $ B $. Then for all $ k $ such that $ 1\leq k\leq n-1 $
	\begin{equation*}
		\dfrac{m_{k+1}}{m_{k}}B=\dfrac{m_{k}}{m_{k+1}}B=B. 		
	\end{equation*}
\end{proposition}
\begin{proof}
	Let $1\leq k\leq n-1$. By formula \eqref{mk+1mk} it holds 
	that $\oslash\Delta <%
	\begin{vmatrix}
		\dfrac{m_{k+1}}{m_{k}}%
	\end{vmatrix}
	\in \pounds$. The fact that $\Delta$ is not an absorber of $B$ and Proposition \ref{tcopt}.\ref{lei3} imply that $ \dfrac{m_{k+1}}{m_{k}} B=B $.
	It follows that $\dfrac{m_{k}}{m_{k+1}}
	B=B  $ for $  1\leq k\leq n-1$.
\end{proof}

When applying the inverse Gauss-Jordan procedure to the right-hand member of the flexible system $ \G^{P}\A|\G^{P}\B $, we define for $ 1\leq q\leq 2n $
$$ [B]^{(-q)}=\left( \left( \G^{P} _q\right)^{-1}\left( \left( \G^{P}_{q+1}\right)^{-1} \cdots \left(  \left( \G^{P}_{2n}\right)^{-1}[B]\right)\right)\right)  \notag $$	
and 
$$ \left( \G^{P}\right) ^{-1}[B] =\left( \left( \G^{P} _1\right)^{-1}\left( \left( \G^{P}_{2}\right)^{-1} \cdots \left(  \left( \G^{P}_{2n}\right)^{-1}[B]\right)\right)\right).$$
\begin{theorem}
	\label{qhnt}  Suppose that the flexible system $ \A|\B $ is  properly arranged with respect to a representative matrix $P$. Assume that $ \Delta $ is not an absorber of $ B $. Then for all $ q $ such that $1\leq q\leq 2n$ one has $ 	[B]^{(q)}= [B]  $ and  $ [B]^{(-q)}=[B] .$ In particular $\G^P[B]=[B]$ and $\left(\G^P\right)^{-1}[B]=[B].$
\end{theorem}

\begin{proof}
	We will apply External Induction. Because $a_{11}=1$, we have $ [B]^{(1)}=\mathcal{G}^{P}_1[B]=I [B]=[B] $.
	
	As for the induction step, let $ q< 2n $ and suppose that $[B]^{(q)}=[B].$
	We consider two cases.
	
	Case 1: $q+1=2k+1$ for some $k\in \{1,\dots,n-1\}$. By the induction
	hypothesis and Proposition \ref{tisogiunguyen} we have 
	\begin{align*}
		&[B]^{(q+1)}=\mathcal{G}^{P}%
		_{q+1} [B]^{(q)}= \mathcal{G}^{P}_{2k+1} [B] \\
		=& 
		\begin{bmatrix}
			1 & 0 & \cdots & 0 & \cdots & 0 \\ 
			0 & 1 & \cdots & 0 & \cdots & 0 \\ 
			\vdots & \vdots & \ddots & \vdots & \ddots & \vdots \\ 
			0 & 0 & \cdots & \dfrac{m_{k}}{m_{k+1}} & \cdots & 0 \\ 
			\vdots & \vdots & \ddots & \vdots & \ddots & \vdots \\ 
			0 & 0 & \cdots & 0 & \cdots & 1%
		\end{bmatrix}%
		\cdot 
		\begin{bmatrix}
			B \\ 
			\vdots \\ 
			B%
		\end{bmatrix}
		=%
		\begin{bmatrix}
			B \\ 
			\vdots \\ 
			B%
		\end{bmatrix}.
	\end{align*}
	Case 2: $q+1=2k+2$ for some $k\in \{1,\dots, n-1\}$. By Theorem %
	\ref{tinhgioihancuasohang}\eqref{cov} all entries of the matrix $ \G^{P}_{2k+2} $ are limited, and the elements of its diagonal are equal to $1$. Then it follows from Case 1 that	
	\begin{equation*}
		[B]^{(q+1)}=\mathcal{G}^{P}%
		_{q+1} [B]^{(q)}= \mathcal{G}^{P}_{2k+2} [B]^{(2k+1)}= \mathcal{G}^{P}_{2k+2} [B]=[B]. 
	\end{equation*}
	In particular $ \G^{P}[B] =[B]^{(2n)}$. This proves the theorem for the Gauss-Jordan procedure $ \G^{P} $. The proof for the inverse procedure is similar.	
\end{proof}

\begin{proposition}\label{Imme1} 
	Suppose that the flexible system $ \A|\B $ is properly arranged with respect to a representative matrix $P$ and stable. Then $$\left(\mathcal{G}^P\right)^{-1}\big(\mathcal{G}^P\B\big)=\B.$$
\end{proposition}
\begin{proof} Let $\B=b+B$. 
	By Proposition \ref{distributivitymatrix} and Theorem \ref{qhnt}, 
	\begin{align*}
		\left(\mathcal{G}^P\right)^{-1}\big(\mathcal{G}^P\B\big)=&\left(\mathcal{G}^P\right)^{-1}\big(\mathcal{G}^P(b+B)\big)
		=\left(\mathcal{G}^P\right)^{-1}\big(\mathcal{G}^Pb+\mathcal{G}^PB\big)\\
		=&\left(\mathcal{G}^P\right)^{-1}\big(\mathcal{G}^Pb+B\big)
		=\left(\mathcal{G}^P\right)^{-1}\big(\mathcal{G}^Pb\big)+\left(\mathcal{G}^P\right)^{-1}B\\
		=&\Big(\left(\mathcal{G}^P\right)^{-1}\mathcal{G}^P\Big)b+B
		=b+B=\B.
	\end{align*} 
\end{proof} 
The neutrix part of the right-hand member is also invariant by multiplying and dividing by the determinants $ \Delta $ and $ \Delta^{(q)} $. This is shown in 
Proposition \ref{divdeltak}.

\begin{proposition}\label{divdeltak} For stable systems $ \A|\B $ it holds that 
	\begin{equation}\label{DeltaB}
		\Delta B=\frac{B}{\Delta}=B.
	\end{equation}
	Moreover, for $ 1\leq q\leq 2n $ the determinant $ \Delta^{(q)} $ is not an absorber of $ B $, and 
	\begin{equation}\label{DeltakB}
		\Delta^{(q)} B=\frac{B}{\Delta^{(q)}}=B.
	\end{equation}
\end{proposition}

\begin{proof}
	It follows from the fact that $ \Delta $ is zeroless and Proposition \ref{cdt} that $ \oslash \Delta <\Delta\subset \pounds $. Also $ \Delta $ is not an absorber of $ B $. Then \eqref{DeltaB} follows from Proposition \ref{tcopt}.\ref{lei3}.  By Proposition \ref{zeroless of Deltak}.\ref{klimited} we have $ \oslash \Delta <\Delta^{(q)}\subset \pounds $. Then also $ \Delta^{(q)} $ is not an absorber of $ B $, hence \eqref{DeltakB} holds by Proposition \ref{tcopt}.\ref{lei3}.
\end{proof}

We are now able to prove that the Gauss-Jordan operations respect the stability property.

\begin{proof} [Proof of Theorem \ref{lmqhnt}] \ref{p1main}. Assume that the system $ \A|\B $ is stable. Let $ 1\leq q \leq 2n $. By Proposition \ref{zeroless of Deltak}\eqref{kstable} the matrix $ \A^{(q)} $ is stable. By Theorem \ref{qhnt} the system $ \A^{(q)}|\B^{(q)} $ is uniform with $ [B]^{(q)}=[B] $. Then $ \Delta^{(q)} $ is not an absorber of $ [B]^{(q)} $ by Proposition \ref{divdeltak}. We still need to show that		
	\begin{equation}\label{stabk}
		R(\A^{(q)})\subseteq R(\B^{(q)}).
	\end{equation}
	Observe that $ R(\A^{(q)}) $ is well-defined, because $ \Delta^{(q)} $ is zeroless by Proposition \ref{zeroless of Deltak}.\ref{kzeroless}.
	
	We show first that for  $ 1\leq q\leq 2n $
	\begin{equation}\label{max neutrix bound}
		\overline{A^{(q) }}\hspace{0.1cm}\overline{\beta^{(q)}}\subseteq B.
	\end{equation}	
	In order to derive \eqref{max neutrix bound}, we show by external induction that  for  $ 0\leq q\leq 2n $ and $1\leq i, j\leq n$
	\begin{equation}	\label{neutrix bound}
		A^{(q)}_{ij}\overline{\beta}^{(q)}\subseteq B.
	\end{equation}
	For $q=0 $ we have by stability, \eqref{rbeta} and \eqref{DeltaB}
	\begin{equation*}
		A^{(0)}_{ij}\overline{\beta^{(0)}}\subseteq\oA \overline{\beta}\subseteq \Delta R(\A) \overline{\beta}\subseteq \Delta R(\B)\overline{\beta}	\subseteq \Delta B=B.
	\end{equation*} 
	Assuming that the property \eqref{neutrix bound} holds for $q<2n$, we will prove it for $q+1$. Because $\overline{\beta^{(q+1)}}=\beta_p^{(q+1)}$
	for some $p\in \{1,\dots, n\}$, 
	\begin{equation*}
		\left|\overline{\beta^{(q+1)}}\right|=\left|\beta^{(q+1)}_p\right|=\left|\sum \limits_{j=1}^n
		g^{(q+1)}_{pj}\beta_j^{(q)}\right|\leq \sum \limits_{j=1}^n
		\left|g_{pj}^{(q+1)}\right|\left|\beta^{(q)}_j\right|\leq \sum \limits_{j=1}^n
		\left|g_{pj}^{(q+1)}\right|\left| \overline{\beta^{(q)}}\right|.
	\end{equation*}
	Also  for $ 1\leq i,j\leq n $
	\begin{equation}  \label{eqdetermimedofA}
		A_{ij}^{(q+1)}=g^{(q+1)}_{i1}A^{(q)}_{1j}+\cdots+g_{in}^{(q+1)}A_{nj}^{(q)}.
	\end{equation}
	If $q+1=2k+2$ for some $k\in \{1, \dots, n-1\}$, by 
	Theorem \ref{tinhgioihancuasohang} and the induction hypothesis one has 
	\begin{alignat*}{2}
		A_{ij}^{(q+1)}\overline{\beta^{(q+1)}}\subseteq &\Big(%
		g^{(q+1)}_{i1}A^{(q)}_{1j}+\cdots+g_{in}^{(q+1)}A_{nj}^{(q)}\Big)\Big(\sum
		\limits_{j=1}^n \left|g_{ij}^{(q+1)}\right| \left| \overline{\beta^{(q)}}\right| \Big) &  \\
		=&\left(g^{(q+1)}_{i1}A^{(q)}_{1j}\overline{\beta^{(q)}}+%
		\cdots+g_{in}^{(q+1)}A_{nj}^{(q)}\overline{\beta^{(q)}}\right)\left(\sum%
		\limits_{j=1}^n \left|g_{ij}^{(q+1)}\right|\right) &  \\
		\subseteq &(\pounds B+\cdots+\pounds B)\pounds\subseteq B. 
	\end{alignat*}	
	If $q+1=2k+1$ for some $k\in \{1, \dots, n-1\}$, we consider separately the cases $i\not=k+1$ and $i=k+1$.
	
	Case 1: For $i\not=k+1$ and $1\leq i\leq n$, the row $g^{(q+1)}_{i} $ is a
	unit vector, so the neutrices of the $i^{th}$ row of $\mathcal{A}^{(q+1)}$ satisfy $%
	A^{(q+1)}_{ij}=A^{(q)}_{ij}$ for $ 1\leq j\leq n $. Also 
	\begin{equation*}
		\beta^{(q+1)}=\left(\beta^{(q)}_1, \dots, \beta^{(q)}_k, \dfrac{m_{k}}{m_{k+1}}%
		\beta^{(q)}_{k+1}, \beta^{(q)}_{k+2}, \dots, \beta^{(q)}_n\right).
	\end{equation*}
	If $\overline{\beta^{(q+1)}}=\beta^{(q)}_s$ for some $s\in \{1, \dots, n\}
	\setminus \{k+1\}$,  for  $i\not=k+1, 1\leq i\leq n$ and $1\leq j\leq
	n $ one has by the induction hypothesis
	\begin{equation*}
		A_{ij}^{(q+1)}\overline{\beta^{(q+1)}} =A_{ij}^{(q)}\beta^{(q)}_s\subseteq
		A_{ij}^{(q)}\overline{\beta^{(q)}}\subseteq B.
	\end{equation*}		
	If $\overline{\beta^{(q+1)}}= \dfrac{m_{k}}{m_{k+1}}\beta^{(q)}_{k+1}$,
	then for $i\not=k+1, 1\leq i\leq n$ and $1\leq j\leq n$ it follows from the induction hypothesis and Proposition \ref{tisogiunguyen} that
	\begin{equation*}
		A_{ij}^{(q+1)}\overline{\beta^{(q+1)}} =A^{(q)}_{ij}\dfrac{m_{k}}{m_{k+1}}%
		\beta^{(q)}_{k+1} \subseteq \dfrac{m_{k}}{m_{k+1}}B=B.
	\end{equation*}%
	
	Case 2: For $i=k+1$, by formula \eqref{eqdetermimedofA} one has for $ 1\leq j\leq n $
	\begin{equation*}
		A^{(q+1)}_{k+1j}=A_{k+1j}^{(q)}\dfrac{m_{k}}{m_{k+1}}.
	\end{equation*}
	If $\overline{\beta^{(q+1)}}=\beta^{(q)}_s$ for some $s\in \{1, \dots, n\}
	\setminus \{k+1\}$, due to Proposition  \ref{tisogiunguyen} one has for $1\leq j\leq n$ 	
	\begin{equation*}
		A_{k+1j}^{(q+1)}\overline{\beta^{(q+1)}} =\dfrac{m_{k}}{m_{k+1}}%
		A_{k+1j}^{(q)}\beta^{(q)}_s\subseteq \dfrac{m_{k}}{m_{k+1}}A_{k+1j}^{(q)}%
		\overline{\beta^{(q)}}\subseteq \dfrac{m_{k}}{m_{k+1}} B=B.
	\end{equation*}		
	If $\overline{\beta^{(q+1)}}= \dfrac{m_{k}}{m_{k+1}}\beta^{(q)}_{k+1}$,
	again using Proposition \ref{tisogiunguyen} we find for  $1\leq j\leq n$ 
	\begin{equation*}
		A_{k+1j}^{(q+1)}\overline{\beta^{(q+1)}} =\dfrac{m_{k}}{m_{k+1}}%
		A^{(q)}_{k+1j}\dfrac{m_{k}}{m_{k+1}}\beta^{(q)}_{k+1} \subseteq \Big(%
		\dfrac{m_{k}}{m_{k+1}}\Big)^2A_{k+1j}^{(q)}\overline{\beta^{(q)}}\subseteq %
		\Big(\dfrac{m_{k}}{m_{k+1}}\Big)^2B=B.
	\end{equation*}		
	Combining, we see that property \eqref{neutrix bound} holds for all $ q $ such that $0\leq q\leq 2n.$ 
	
	Formula \eqref{max neutrix bound} follows directly from \eqref{neutrix bound}.
	
	To finish the proof, we consider separately the cases that $\overline{\beta^{(q) }}$ is zeroless and that $\overline{\beta^{(q)}}=B$ is neutricial. 
	If $\overline{\beta^{(q) }}$ is zeroless, 
	by \eqref{max neutrix bound} and Proposition \ref{divdeltak}
	\begin{equation*}
		R(\mathcal{A}^{(q)})=\frac{\overline{A^{(q)}}}{\Delta^{(q)}}\subseteq\frac{1}{\Delta^{(q)}}\frac{B}{\overline{\beta^{(q)}}}=\frac{B}{\overline{\beta^{(q)}}}=R(\B^{(q)}).
	\end{equation*}		
	If 
	$\overline{\beta^{(q)}}=B$ is neutricial, formula \eqref{max neutrix bound} takes the form 	$\overline{A^{(q)}} B\subseteq 	B$. Then 
	$$R(\mathcal{A}^{(q)})B=%
	\overline{A^{(q)}}\frac{B}{\Delta^{(q)}}=\overline{A^{(q)}}B\subseteq B.$$
	We conclude from Theorem \ref{qhnt} that $R(\mathcal{A}^{(q)})\subseteq B:B=R([B^{{(q)}}])=R(\B^{{(q)}})$.
	
	\ref{p2main}. By setting $ q=2n $ we obtain from Part \ref{p1main} that the final system $\mathcal{G}^P\mathcal{A})|\G^P\mathcal{B}$ is stable, while	$ \G^{P}A $ is a near-identity matrix by Theorem \ref{thenearid}.
\end{proof}

\section{Stability and Cramer's rule}\label{Cramer}

By Theorem 4.4 of \cite{Jus} Cramer's rule in the form  \eqref{forcra} solves non-singular reduced uniform non-homogeneous stable systems. Here we extend the proof to homogeneous systems. We continue to adopt Convention \ref{conv} (the fact that $ \A $ is properly arranged is not essential here) and prove the following theorem.

\begin{theorem}[Cramer's Rule for flexible systems]\label{thecra}
	If the system $ \A|\B $ is stable, its solution is given by the external vector \eqref{forcra}.
\end{theorem}

Let $ \xi $ be given by \eqref{forcra}. Proposition \ref{nhtt} shows that the neutrix parts of the components of $ \xi $ are equal to the neutrix at the right-hand side $ B $. Then the proof of Theorem \ref{thecra} for homogeneous systems consists in showing that the solution is neutricial, with components equal to $ B $. 

We first introduce some notations, which in part will be used in the proof that for stable systems the Gauss-Jordan solution and the Cramer solution coincide, and provide bounds for the determinants of $ M_j $ and its neutrices.



\begin{definition}\label{defcra2}
	Consider the system $ \A|\B $ with $ \A\in\M_{n} (\E)$ non-singular. Let $ \Delta=\det(\A)\equiv d+D $ with $d=\det(P)$, where $P$ is a representative matrix of $\A$. For $1\leq i\leq n $, let $M_i(b) $ be the matrix obtained from $\A$ by the substitution of the $i^{th}$ column by a representative vector $ b $ of $ \B $. We write
	\begin{equation*}\label{forcrabd}
		\xi(b,d)^{T}=\left(\dfrac{\det (M_1(b))}{d},\ldots,%
		\dfrac{\det (M_n(b))}{d}\right)^{T}
	\end{equation*}
	\begin{equation*}\label{forcrab}
		\xi(b)^{T}=\left(\dfrac{\det (M_1(b))}{\Delta},\ldots,%
		\dfrac{\det (M_n(b))}{\Delta}\right)^{T}.
	\end{equation*}	
\end{definition}

\begin{lemma}
	\label{bd3.1} Assume the system $ \A|\B $ is stable.  Then for $ 1\leq j\leq n $
	\begin{enumerate}
		\item \label{bode3.1i} $\left|\det (M_j)\right|\leq 2n!\left|\overline{\beta}\right|$.
		\item \label{bode3.1ii} $N\big(\det (M_j)\big)\subseteq \overline{\beta}\cdot \oA+B.$
	\end{enumerate}
\end{lemma}
\begin{proof}	
	Let $S_n$ be the set of all  permutations of $\{1,\dots, n\}$ and $\sigma\in S_n$. Put  $$\gamma_\sigma=\alpha_{\sigma(1)1}\dots\alpha_{\sigma(j-1)j-1}\alpha_{\sigma(j+1)j+1}\dots\alpha_{\sigma(n)n}.$$ 
	Because the system is reduced,  \begin{equation}
		\label{danh gia phan tu dinh thuc}
		|\gamma_\sigma|\leq \overline{\alpha}^{n-1}\leq (1+\oslash)^{n-1}=1+\oslash,
	\end{equation}	
	and, as a consequence of Proposition \ref{tcopt}.\ref{tcoptv}, 
	\begin{equation}\label{Ngamma}	
		N(\gamma_\sigma)=N\left( \prod_{1\leq k\leq n, k\neq j}(a_{\sigma(k)k}+A_{\sigma(k)k})\right) \subseteq N(1+\overline{A})^{n-1}=\overline{A}.	
	\end{equation}	
	\ref{bode3.1i}. It follows from \eqref{danh gia phan tu dinh thuc} that
	$$\left|\det (M_j)\right|\leq \sum\limits_{\sigma \in S_n}\left|\gamma_\sigma\beta_{\sigma(j)}\right|\leq \sum\limits_{\sigma \in S_n}\left|(1+\oslash) |\overline{\beta}|\right|= n!(1+\oslash) |\overline{\beta}|=2n! |\overline{\beta}|.$$ 
	
	\ref{bode3.1ii}. It follows from \eqref{Ngamma} and \eqref{danh gia phan tu dinh thuc} that	
	\begin{alignat*}{2}
		N\Big(\det \big(M_j\big)\Big)&=N\left(\sum\limits_{\sigma\in S_n}\mbox{sgn}\left(\sigma\right)\gamma_\sigma \beta_{\sigma(j)}\right)=\sum\limits_{\sigma\in S_n}N\left(\gamma_\sigma \beta_{\sigma(j)}\right)\\
		&=\sum\limits_{\sigma\in S_n}\left( \beta_{\sigma(j)}N(\gamma_\sigma)+\gamma_\sigma N(\beta_{\sigma(j)})\right) \\
		&\subseteq \sum\limits_{\sigma\in S_n}\left( \overline{\beta} \oA + (1+\oslash) B\right)  =\overline{\beta} \overline{A} + B.
	\end{alignat*}
\end{proof}

\begin{proposition}\label{nhtt} Assume the system $ \A|\B $ is stable.  Then for $1\leq j\leq n$
	\begin{equation*}  
		N\left(\dfrac{\det (M_j)}{\Delta}\right)=B.
	\end{equation*}	
	As a consequence, if the system is homogeneous, for $1\leq j\leq n$
	\begin{equation*} 
		\dfrac{\det (M_j)}{\Delta}=B.
	\end{equation*} 
\end{proposition}
\begin{proof} Let $D=N(\Delta)$. By Proposition \ref{tcopt}.\ref{tcoptix},  Lemma \ref{bd3.1}  and Proposition \ref{cdt}, we have for $1\leq j\leq n$  
	\begin{alignat}{2}\label{cong thuc moi chuong 4 }N\left(\dfrac{\det\big( M_j\big)}{\Delta}\right)=&\dfrac{1}{\Delta}N\left(\det (M_j)\right)+\det\big( M_j\big) N\left(\dfrac{1}{\Delta}\right) \\ 
		=&\dfrac{1}{\Delta}N\left(\det \big(M_j\big)\right)+\det \big(M_j\big) \dfrac{D}{\Delta^2}\notag\\
		\subseteq & \dfrac{1}{\Delta}(\overline{\beta} \oA+B)+2n!\overline{\beta} \dfrac{D}{\Delta^2}\notag\\
		\subseteq & \dfrac{\overline{\beta} \oA}{\Delta}+\dfrac{B}{\Delta}+\overline{\beta} \dfrac{\overline{A}}{\Delta^2}.\notag 
	\end{alignat}
	From the stability condition $R(\A)\subseteq R(\B)$ we derive both in the homogeneous and  non-homogeneous case that $\overline{\beta} \dfrac{\overline{A}}{\Delta}\subseteq B. $ Then we  obtain from \eqref{cong thuc moi chuong 4 } and Proposition \ref{divdeltak}  that
	\begin{equation} \label{nam1} 
		N\left(\frac{\det\big( M_j\big)}{\Delta}\right)\subseteq \overline{\beta}\frac{  \oA}{\Delta}  +\frac{B}{\Delta} +\frac{1}{\Delta}\left(\frac{\overline{\beta}}{\Delta}\oA\right) \subseteq B+B+B/\Delta=B.
	\end{equation}
	It follows from Proposition \ref{lei2} that  $\left|\Delta_{ij}\right|>\oslash\Delta$ for some $i\in \{1,\dots, n\}$.  Because $\Delta$ is not an absorber of $B$, also $ \Delta_{ij} $ is not an absorber of $B$. Hence $B\subseteq B \Delta_{ij}$. Using the fact that products containing a neutrix have always the same sign and subdistibutivity, we derive that
	\begin{alignat*}{2}\label{dttv2}
		B& \subseteq B\Delta_{1j}+\cdots+ B \Delta_{nj}\\
		&\subseteq\det\begin{bmatrix} 1+A_{11} & \cdots & \alpha_{1(j-1)} & B & \alpha_{1(j+1)}& \cdots & \alpha_{1n}\\
			\vdots & \ddots & \vdots & \vdots & \vdots & \ddots & \vdots\\
			\alpha_{n1} & \cdots & \alpha_{n(j-1)} & B & \alpha_{n(j+1)}& \cdots & \alpha_{nn}
		\end{bmatrix}\\
		&\subseteq   N\left(\det\begin{bmatrix} 1+A_{11} & \cdots & \alpha_{1(j-1)} & b_1+B & \alpha_{1(j+1)}& \cdots & \alpha_{1n}\\
			\vdots & \ddots & \vdots & \vdots & \vdots & \ddots & \vdots\\
			\alpha_{n1} & \cdots & \alpha_{n(j-1)} & b_n+B & \alpha_{n(j+1)}& \cdots & \alpha_{nn}
		\end{bmatrix}\right) \\
		&= N\left(\det (M_j)\right). 
	\end{alignat*}
	Then by Proposition \ref{divdeltak} 
	\begin{equation}\label{nam2}
		B=\dfrac{B}{\Delta}=\dfrac{N\big(\det (M_j)\big)}{\Delta}\subseteq \dfrac{N\big(\det (M_j)\big)}{\Delta} +\det(M_j) N\left(\dfrac{1}{\Delta}\right)=N\left(\dfrac{\det M_j}{\Delta}\right).
	\end{equation} 
	Combining \eqref{nam1}  and \eqref{nam2}, we conclude that $B=N\left(\dfrac{\det (M_j)}{\Delta}\right)$ for $ 1\leq j\leq n$.
	
	As a consequence, if the system is homogeneous, it holds that $\dfrac{\det (M_j)}{\Delta}=N\left(\dfrac{\det (M_j)}{\Delta}\right)=B$ for $ 1\leq j\leq n$.
\end{proof}

\begin{proof}[Proof of Theorem \ref{thecra}]
	Let $x=(x_1, x_2, \dots, x_n)^T\in \xi$. In order to show that $ x $ satisfies the system $ \A|\B $, assume first that  the system is not homogeneous. By Theorem 4.4 of \cite{Jus} the external vector $\xi$ given by \eqref{forcra} is the  solution of the system $ \A|\B $ given in the form \eqref{hpttq}, hence $ x $ satisfies $ \A|\B $ by Proposition \ref{extreal}. 
	Secondly, assume that the system $ \A|\B $ is homogeneous. Then $\xi=(B, B,\dots, B)^T$ by Proposition \ref{nhtt}. By direct verification we see that $\xi$ satisfies \eqref{hpttq}. Again $x$ is a solution of the system $ \A|\B $ by Proposition \ref{extreal}. 
	
	Suppose now that $x$ is an admissible solution of system $ \A|\B $. Let $ P=[a_{ij}]_{n\times n} $ be a representative matrix for $ \A $. Then for $1\leq i\leq n$ there exists $b_i\in \beta_i$ such that 
	$$\left\{\begin{array}{cccccl}
		a_{11}x_1+&\cdots&+a_{1n}x_n&=&b_1\\
		\vdots &\ddots & \vdots& \vdots&\\
		a_{n1}x_1+&\cdots& +a_{nn}x_n&=&b_n
	\end{array}.\right. $$
	Let $ b=(b_{1},\dots, b_{n})^T $. 
	By Cramer's rule, one has $x_i=\dfrac{M_i^{P}(b)}{d}\in \dfrac{\det(M_i)}{\Delta}$ for $1\leq i\leq n$. Hence $ x\in \xi $.
\end{proof} 

\section{Proof of the Main Theorem}\label{proofs}

The proof of Theorem  \ref{maintheorem} is organized as follows: we prove first that the Gauss-Jordan procedure does not alter the solution of a stable system, i.e. the set of admissible solutions.  With this and Theorem \ref{thecra}, we show that the solutions given by Cramer's rule and by Gauss-Jordan elimination are equal. Then we show that a stable system such that the coefficient matrix is a near-identity matrix is simply solved by the right-hand member. The Main Theorem will follow from these theorems. 


\begin{theorem}\label{idltdng}
	Suppose that the system $ \A|\B $ is stable. Then the Gauss-Jordan solution $ G $ is well-defined, and a real vector $ x $ is an admissible solution if and only if $ x\in G $.
\end{theorem}

\begin{proof} Let $ S $ be the solution of $ \A|\B $, and $ P $ be a properly arranged matrix of representatives of $\A$.  
	
	Assume first that $ x  $ is an admissible solution, i.e. $ x\in S $. Then $ \mathcal{A}x\subseteq\B $. By Proposition \ref{combinetheorem} and Proposition \ref{matincl}
	\begin{equation*}
		\Big(\mathcal{G}^P(\mathcal{A})\Big)x = \mathcal{G}^P(\A x)\subseteq \mathcal{G}^P(\mathcal{B}).
	\end{equation*}
	Hence $ x\in G^{P} $.
	
	Conversely, assume that $ x\in G^{P} $. Then $ \Big(
	\mathcal{G}^P(\mathcal{A})\Big)x \subseteq \mathcal{G}^P(\mathcal{B})$. Using Proposition \ref{subassociativity}, Proposition \ref{combinetheorem} and Proposition \ref{Imme1} we derive that
	\begin{align*}
		\mathcal{A}x=I(\mathcal{A}x)&=(\left(\mathcal{G}^P\right)^{-1}\mathcal{G}^P)(\mathcal{A}x)\\
		&\subseteq \left(\mathcal{G}^P\right)^{-1}\Big(\mathcal{G}^P(\mathcal{A}x)\Big)=\left(\mathcal{G}^P\right)^{-1}\Big(\Big(\mathcal{G}^P(\mathcal{A})\Big)x\Big)\\
		&\subseteq\left(\mathcal{G}^P\right)^{-1}\big(\mathcal{G}^P(\B)\big)=\B.
	\end{align*}
	Hence $ x\in S $. Combining, we see that $ S=G^{P} $. Consequently  $ G^{P} $ does not depend on the choice of $ P $, hence $ G\equiv G^{P} $ is well-defined. We conclude that  $ S=G $.
\end{proof}



\begin{theorem}
	\label{dlng3} Assume that the system $ \A|\B $ is stable. Then the Gauss-Jordan solution is equal to the Cramer-solution.
\end{theorem}

\begin{proof}
	The theorem follows from  Theorem \ref{idltdng}  and Theorem \ref{thecra}.
\end{proof}

The solution of the system whose coefficient matrix is the identity matrix is of course the right-hand member. We use Theorem \ref{thecra} to show that this property remains valid if the coefficient matrix is a near-identity matrix, provided the system is stable.
\begin{theorem}\label{Cramer-solution of near identity system}
	Let $ \A $ be a near-identity matrix and $\B=b+B$. Suppose that the system $\A|\B$ is stable. Then $\B$ is the solution of the system.
\end{theorem}
\begin{proof} Put $	\xi=(\xi_1,\dots, \xi_n)^T$ with $ \xi_i=\det (M_i)/\Delta $ for $1\leq i\leq n$. By Theorem \ref{thecra} the vector $ \xi $ is the solution of the system $\A|\B$. We have $\A=I_{n}+A$ with $A\subseteq [\oslash]_{n\times n}$, so $ I_{n} $ is a representative matrix of $ \A $, and $ b_{i} $ is a representative of $ \det (M_i)$ for $1\leq i\leq n$. It follows from the stability that $\Delta=1+D$ with $D\subseteq \oA\subseteq \oslash$. In addition, by Proposition \ref{nhtt} it holds that $N\left(\dfrac{\det (M_i)}{		\Delta}\right)=B$ for $1\leq i\leq n$. Then
	\begin{equation*}  \label{nghiem gauss2n}
		\xi_i= 		b_{i}+N\left( \dfrac{\det (M_i)}{\Delta}\right)=b_i+B=\beta_{i}.
	\end{equation*}Hence
	for $  1\leq i\leq n $, i.e. $ \xi=\B $.  Hence the   solution of the system $\A|\B$ is equal to $\B$.
\end{proof}
Theorem \ref{TheoremrepresentGausssolution} gives an effective way to find the solution. As in the real case, the solution of $\A|\B$ is given by the Gauss-Jordan procedure, where by Theorem \ref{idltdng} we may choose any representative matrix $ P $ of $ \A $, provided it is reduced and properly arranged. The result follows from the fact that the Gauss-Jordan procedure, which due to Part \ref{p2main} of Theorem \ref{limitofelement} 
does not affect the stability of the system, leads to a stable system whose matrix of coefficients is a near identity matrix, the  solution of which is equal to right-hand member by Theorem \ref{Cramer-solution of near identity system}. 

\begin{theorem}
	\label{TheoremrepresentGausssolution} Suppose that the system $ \A|\B $ is stable, and properly arranged with respect to a  representative matrix $P$ of $\A$. Then $\mathcal{G}^{P}(\mathcal{B})$ is the Gauss-Jordan solution of $ \A|\B $.
\end{theorem}
\begin{proof}
	
	By Theorem \ref{thenearid} it holds that  $\G^P(\mathcal{A})$ is a near-identity matrix.
	By Part \ref{p2main} of Theorem \ref{lmqhnt},  the system $\Big(\G^P(\A)\Big)x\subseteq \G^P(\B)$ is stable.
	Then $\G^P(\B)$  is the solution of the system $\Big(\G^P(\A)\Big)x\subseteq \G^P(\B)$ by Theorem \ref{Cramer-solution of near identity system}. Hence  $ \G^P(\B)=\G^P $, so it is the Gauss-Jordan solution of the system $\A|\B$ with respect to $ P $. By Theorem \ref{idltdng} it is Gauss-Jordan solution of the system $\A|\B$.
\end{proof}

\begin{proof}[Proof of Theorem \ref{maintheorem}]
	The solution $ S $ is equal to the Gauss-Jordan solution $ G $ by Theorem \ref{idltdng}, which also says that the application of the Gauss-Jordan procedure does not depend on the choice of  the matrix of representatives $ P $. Then $ G=\G^{P}(\B) $ by Theorem \ref{TheoremrepresentGausssolution}. Also $ G $ is equal to the Cramer-solution by Theorem \ref{dlng3}, which takes the form \eqref{forcra} by Theorem \ref{thecra}.
\end{proof} 
\section{Equivalent systems}\label{neglection of terms}

Systems with the same right-hand member will be said to be equivalent if they have equal solutions. By showing that two systems are equivalent, we may obtain simplifications. In particular, let $ \A|\B $ be a system such that some of the entries of $ \A $ are given in the form of expansions. Assume that $ \A' $ is obtained from $ \A $ by truncating the expansions in such a way that $ \A'|\B $ is equivalent to $ \A|\B $. Then we may solve as well the simplified system $ \A'|\B $,  neglecting the extra terms occurring in $ \A $. 

If $ \A|\B $ is stable, we will see that the simplification is justified if the neglected terms $ t $ of the expansions satisfy $ t/\Delta\subseteq R(\B)$. We may roughly interpret this by the possibility to neglect decimals in a coefficient matrix, if compared with the determinant they are small with respect to the relative imprecisions of the right-hand member. 

We will illustrate the effects of simplification with the help of Example \ref{ex6} and Example \ref{ex7} below, and some numerics. Again we consider systems $ \A|\B $ in the sense of Convention \ref{conv}.

\begin{definition}\label{defequiv}
	Let $\mathcal{A},\mathcal{A'}\in \M_{n}(\E)$ and $ \B\in \M_{n,1}(\E) $ be an external vector. The system $ \A'|\B $ is said to be equivalent to $ \A|\B $ if the solution of $ \A'|\B $ is equal to the solution of $ \A|\B $.
\end{definition}

 Proposition \ref{samesol} gives conditions for such flexible systems to be equivalent.
\begin{proposition}\label{samesol}
	Let $ \A|\B $ be a stable system with solution $ S=\G^{P}(\B) $, where $ P=[a_{ij}]_{n\times n} $ is a reduced properly arranged representative matrix. Let $ Q=[q_{ij}]_{n\times n}\in\M_{n}(\R) $ be a reduced properly arranged matrix such that for $ 1\leq i,j \leq n $
	\begin{equation}\label{difqa}
		q_{ij}-a_{ij}\in \oA.
	\end{equation}
	Let $ \A'\equiv [\alpha'_{ij}] $ with $ \alpha'_{ij}=q_{ij}+A'_{ij} $ and $ A'_{ij}\subseteq\oA $ for $ 1\leq i,j \leq n $. Then $ \A'|\B $ is a stable equivalent system, and $ \G^{P}(\B)=\G^{Q}(\B) $.
	
\end{proposition}

We prove first two lemmas. 

\begin{lemma}\label{stns}
	Let $ \A\in \M_{n}(\E)=[\alpha_{ij}]_{n\times n} =[a_{ij}+A_{ij}]_{n\times n}$ be a non-singular stable matrix, properly arranged with respect to a reduced representative matrix $P=[a_{ij}]_{n\times n}  $. Let $ \A'\equiv [\alpha'_{ij}]_{n\times n} $ be defined by
	\begin{equation*}
		\alpha'_{ij}=a_{ij}+A'_{ij}, 
	\end{equation*}
	with $ A'_{ij}\subseteq\oA $ for $ 1\leq i,j \leq n $.
	Then the matrix $ \A' $ is non-singular and stable.
\end{lemma}
\begin{proof}
	Because $ \A $ is limited, the matrix $ \A' $ is also limited. Let $ d=\det (P)$ and $ \Delta'=\det(\A') $. Then $ d $ and $ \Delta' $ are also limited. Because $A'_{ij}\subseteq\oA $ for $ 1\leq i,j \leq n $, and the matrix $ \A $ is non-singular and stable, it holds that $ \Delta'\subseteq d+\oA\subseteq (1+\oslash)d$. So $ \Delta' $ is zeroless, hence $ \A' $ non-singular. In addition 
	
	$$ \frac{\overline{\A'}}{\Delta'} \subseteq \frac{\overline{\A}}{(1+\oslash)d}=\frac{\overline{\A}}{\Delta}\subseteq\oslash.$$
	Hence $ \A' $ is stable.
\end{proof}
\begin{lemma}\label{steqs}
	Let $ \A|\B $ be a stable system, and $ P=[a_{ij}]_{n\times n} $ be a reduced properly arranged representative matrix of $ \A $. 
	Let $ \A'\equiv [\alpha'_{ij}] $ with $ \alpha'_{ij}=a_{ij}+A'_{ij} $ and $ A'_{ij}\subseteq\oA $ for $ 1\leq i,j \leq n $. Then $ \A'|\B $ is a stable equivalent system satisfying Convention \ref{conv}.
\end{lemma}
\begin{proof}
	Because $ \overline{\A'}\subseteq\oA $, by Lemma \ref{stns} the matrix $ \A' $ is non-singular and stable. Then $ \A'|\B $  satisfies Convention \ref{conv}. By Theorem \ref{TheoremrepresentGausssolution} both systems $ \A|\B $ and $  \A'|\B $ are solved by $ \G^P(\B) $. Hence the systems are equivalent.
\end{proof}
\begin{proof}[Proof of Proposition \ref{samesol}]
	Put $ \A''=Q+(\oA)_{n\times n} $; note that the system $  \A''|\B $ satisfies Convention \ref{conv}. It follows from \eqref{difqa} that $\A''=P+(\oA)_{n\times n} $, so by Lemma \ref{stns} the matrix $ \A''$ is non-singular and stable. Then $  \A''|\B $ is a stable system, and by Lemma \ref{steqs} the systems $  \A|\B $ and $  \A''|\B $ are equivalent. The system $  \A'|\B $ shares with the system $  \A''|\B $ the representative matrix $ Q $, hence by Lemma \ref{steqs} it is stable and equivalent to $  \A''|\B $. Hence the systems $  \A|\B $ and $  \A'|\B $ are also equivalent. Then it follows from Theorem \ref{TheoremrepresentGausssolution} that $\G^{P}(\B)= \G^{Q}(\B) $. 
\end{proof}

\begin{corollary}\label{eqrepmat}
	Let $ \A|\B $ be a stable system, where $ \A=P+A $, with $ P=[a_{ij}]_{n\times n} $ is a reduced properly arranged representative matrix and $ A $ a neutricial matrix. Let $ Q=[q_{ij}]_{n\times n}\in\M_{n}(\R) $ be a reduced properly arranged representative matrix of $ \A'\equiv P + (\oA)_{n\times n}$. Then $ \A|\B $ and $ Q|\B $ are equivalent.
\end{corollary}

The corollary indicates that for stable systems we may neglect all terms and neutrices smaller than the biggest neutrix in the coefficient matrix, and solve instead for any real coefficient matrix lying within this range of imprecision.

We will apply Proposition \ref{samesol} and Corollary \ref{eqrepmat} in Example \ref{ex6n} and Example \ref{ex7} below. 
\begin{example}	\label{ex6n}	\rm 
	(Continuation of Example \ref{ex6}.) Put
	\begin{equation*}\label{sy6n}
		\left\{ 
		\setlength\arraycolsep{0pt}
		\begin{array}{ r >{{}}c<{{}} r >{{}}c<{{}} r  @{{}\subseteq {}} r  >{{}}c<{{}} r  >{{}}c<{{}}  r }
			\left( 1+\varepsilon^2\oslash \right) x _{1} &+ &(1+\varepsilon^2\oslash)x _{2} &+& \left( 	1+\varepsilon^2\oslash \right) x _{3}& 1+%
			\varepsilon \oslash \\ 
			\left( 1+\varepsilon^2\oslash \right) x _{1}&+&\left( -\frac{1}{2}+\varepsilon^2\oslash	\right) x _{2}&+&(-\frac{1}{2}+\varepsilon^2\oslash)x _{3}& -2+\varepsilon \oslash \\ 
			\left( \frac{1}{2}\varepsilon +\varepsilon^2\oslash\right) x _{1}&+&(\frac{1}{2}+\varepsilon^2\oslash) x _{2}&+&\left(
			1+\varepsilon^2\oslash \right) x _{3}& \varepsilon+\varepsilon \oslash %
		\end{array}.
		\right.
	\end{equation*}
	Let $ \A' $ be the coefficient matrix of \ref{ex6n}. Note that the matrix $P$ given by \eqref{repmatrixex} is a representative matrix of both $A$ and $\A'$, and that the associated neutricial matrices satisfy $A\subseteq A'$ with $ \overline{A'}=\oA $, and that the vectors in the right-hand side of both systems are the same. Then by Proposition \ref{samesol}
	the solution of the system \eqref{ex6n} is also given by \eqref{sol1}.
\end{example}

\begin{example}
	\label{ex7}%
	\rm%
	Let $\varepsilon>0 $ be infinitesimal. Consider the reduced flexible system $ \A|\B $ given by
	\begin{equation}\label{eq7}
		\left\{ 
		\setlength\arraycolsep{0pt}
		\begin{array} {lllllllllllll}
			\left( 1+\varepsilon \pounds \right) x_{1}  + (1-\eps) x_{2} +(\frac{1}{2}+2\eps^2) x
			_{3}+\frac{1}{2}x_{4}& \subseteq &-1&+&\varepsilon \pounds  \\ 
			(-1+3\eps)x_{1}+x_{2}+\left( \frac{1}{2}+\eps^2+\varepsilon^{2} \oslash \right)
			x_{3}+\frac{1}{2}x_{4}& \subseteq & & & \varepsilon \pounds  \\ 
			x_2 -\frac{1}{2}x_{3}+(1-3\eps^2+\varepsilon^{2} \oslash)x_{4}&\subseteq &-\frac{1}{2}&+&\varepsilon \pounds  \\ 
			\left( \frac{1}{2}+\eps+\varepsilon \oslash \right) x_{1}+(1+ \varepsilon\pounds)x_3 +(1+\varepsilon\oslash)x_4 &\subseteq &2  &+& \varepsilon \pounds %
		\end{array}.
		\right.
	\end{equation}	
	The  matrix
	\begin{equation}\label{Qrep}	
		P=\begin{bmatrix}
			1&1-\eps&1/2+2\eps^{2}&1/2\\
			-1+3\eps&1&1/2+\eps^{2}&1/2\\
			0&1&-1/2&1-3\eps^{2}\\
			1/2+\eps &0&1&1
		\end{bmatrix}, 
	\end{equation}
	is a representative matrix of $ \A $. One verifies that $ P $ is non-singular, reduced and properly arranged, with $\det(P)\in 3+\pounds\eps$ zeroless, $ m_{1}=1 $, $ m_{2}\in 2-2\eps+\oslash\eps $ and $ m_{3}\in 2-5/2\eps+\oslash\eps $.
	Also
	$R\left( \mathcal{A}%
	\right) =\overline{A}\diagup \Delta =\varepsilon \pounds $, $R\left( \mathcal{B}%
	\right) =B\,\diagup \overline{\beta }=\varepsilon \pounds $ and $\Delta 
	B=\varepsilon \pounds =B$. Hence $\oA=\varepsilon\pounds\subset\oslash=\oslash\Delta$, $R\left( \mathcal{A}%
	\right) \subseteq R\left( \mathcal{B}\right) $ and $\Delta $\ is not an
	absorber of $B$, so the system $ \A|\B$ stable. 	 
	
	Let \begin{equation}\label{repmat4}
		Q=\begin{bmatrix}
			1&1&1/2&1/2\\
			-1&1&1/2&1/2\\
			0&1&-1/2&1\\
			1/2&0&1&1
		\end{bmatrix}.
	\end{equation} 
	The matrix $ Q $ is reduced and non-singular, with determinant $ d\equiv \det(Q)=-3 $. A straightforward calculation shows that $Q$ is properly arranged, with	$ m_{2} =2 $ and $ m_{3} =-2 $. The entries of $ Q $ and $ P $ differ for at most a limited multiple of $ \eps $, which is contained in $\oA=\pounds \eps$.
	
	Applying the usual Gauss-Jordan procedure we derive that
	\begin{equation*}\label{sol4}
		\begin{array}{lll}
			\mathcal{X}=	\mathcal{G}^Q\B =\begin{bmatrix}
				-1/2+\eps\pounds\\-13/8+\eps\pounds\\3/4+\eps\pounds\\3/2+\eps\pounds
			\end{bmatrix}
		\end{array}	
	\end{equation*}	 	
	is  the Gauss-Jordan solution of the system. By Corollary \ref{eqrepmat} it is also the solution of \eqref{eq7}.
\end{example}

We illustrate Example \ref{ex6}/\ref{ex6n} and Example \ref{ex7} numerically. We  assume that $\varepsilon =0.01$, and represent $\oslash $ by $[-0.1,0.1]$ and $\pounds $ by the interval $[-2,2]$. We will not do an exhaustive  investigation, and instead of applying interval calculus  we choose the extreme values of the numerical intervals somewhat at random.

Working with the matrix \eqref{repmatrixex} and the right-hand member $ (1,-2,1/100)^{T} $, we find the exact solution 
\begin{equation*}
	x=\left[ \begin{array}{c}
		x_{1} \\ 
		x_{2} \\ 
		x_{x}%
	\end{array}\right] 
	=
	\left[ \begin{array}{c}
		-1 \\ 
		\frac{397}{100} \\ 
		\frac{197}{100}
	\end{array}\right] %
	=%
	\left[ \begin{array}{c}
		-1 \\ 
		3.97 \\ 
		-1.97%
	\end{array}\right] .%
\end{equation*}

To represent the coefficient matrix of Example \ref{ex6}, we may consider, say, 

\begin{equation*}
	\mathcal{A}^{\prime }=%
	\left[ \begin{array}{ccc}
		1.00001 & 0.99999 & 1.000002 \\ 
		0.999998 & -0.50001 & -0.5 \\ 
		0.00499999 & 0.5 & 1.00001%
	\end{array}\right].
\end{equation*}
Rounded off at $7$ significative digits, we find the the solution%
\[
x^{\prime }=\left[
\begin{array}{r}
	-0.999970\\ 
	3.969929 \\ 
	-1.969945%
\end{array}\right]%
.
\]
The largest deviation with respect to the exact solution is about $0.000071$
in the second coordinate, which is significantly smaller than $0.001$, i.e. the absolute value of the bounds of the interval representing $\oslash \varepsilon $.

In Example \ref{ex6n} all entries of the coefficient matrix are imprecise.  In order to compare with the numerical matrix $ \mathcal{A}^{\prime } $, we choose a matrix $ \mathcal{A}^{\prime \prime  }$ using a randomization which is the same
for the imprecise coefficients of $ \mathcal{A}^{\prime }$ and put
\[
\mathcal{A}''=%
\left[\begin{array}{ccc}
	1.00001 & 0.99999 & 1.00001 \\ 
	0.99999 & -0.50001 & -0.49999 \\ 
	0.004999 & 0.49999 & 1.00001%
\end{array}\right].
\]%
Rounded off at $7$ significative digits, we find the the solution 
$$x^{\prime \prime }=
\left[ \begin{array}{r}
	-0.999943\\ 
	3.969928 \\ 
	-1.969915
\end{array}\right]. %
$$%
As expected, the result is not as good as $ x' $, still the largest
deviation of about $ 0.000085 $ for the third coordinate lies well within the interval $[-0.001,0.001]$ representing $\oslash \varepsilon $.

Finally we illustrate Corollary \ref{eqrepmat} by comparing the solution of the system \eqref{eq7} when using the representative matrices $ P $ given by \eqref{Qrep} and $ Q $ given by \eqref{repmat4}.

The solution $ x'$ for the matrix 
\begin{equation*}
	P'=\begin{bmatrix}
		1&0.99&0.5002&0.5\\
		-0.97&1&0.5001&0.5\\
		0&1&-0.5&0.9997\\
		0.51 &0&1&1
	\end{bmatrix}
\end{equation*} 
is, rounding off at $ 7 $ significative digits,
\begin{equation*}
	\begin{array}{lll}
		x'=	\mathcal{G}^{P'}b =
		\left[ \begin{array}{r}
			-0.5159373\\-1.632099\\0.7537178\\1.509410
		\end{array}\right].
	\end{array}
\end{equation*} 
For the matrix $ Q $ and the right-hand member $ b^{T}= (-1,0, -1/2,2)^{T}$ we find the exact solution
\begin{equation*}
	\begin{array}{lll}
		x=	\mathcal{G}^Qb =\begin{bmatrix}
			-1/2\\-13/8\\3/4\\3/2
		\end{bmatrix}=
		\begin{bmatrix}
			-0.5\\-1.625\\0.75\\1.5
		\end{bmatrix}
	\end{array}.
\end{equation*}

We observe the largest deviation between $ x' $ and $ x $ in the second coordinate, with a value of about $0.007$. This is $ 0.7 $ times the value $ 0.01 $ chosen for $ \varepsilon $, so it can be considered to lie within $ \pounds \varepsilon $.

\end{document}